\documentclass{amsart}
\usepackage[latin1]{inputenc}
\usepackage{ae,aecompl,amsbsy,amssymb,amsmath,amsthm,
eurosym,amsfonts,epsfig,graphicx,graphics,verbatim,enumerate,esint,color,MnSymbol}
\usepackage{url}
\usepackage{color,soul}
\usepackage{longtable}
\usepackage{multirow}
\definecolor{lightblue}{rgb}{.85,.93,1}
\sethlcolor{lightblue}
\newcommand{\mbar}{\bar{M}}
\newcommand{\cdotroomy}{\,\cdot\,}
\newcommand{\llq}{L^{q}_{D}(\omega)}
\newcommand{\lqp}{L^{q'}_{D^*}(\omega)}
\newcommand{\lp}{L^{p}_{C}(\sigma)}
\newcommand{\lpp}{L^{p'}_{C^*}(\sigma)}
\newcommand{\domega}{\,\mathrm{d}\omega}
\newcommand{\dsigma}{\,\mathrm{d}\sigma}

\newcommand{\dmu}{\,\mathrm{d}\mu}
\newcommand{\dx}{\,\mathrm{d}x}

\newcommand{\angles}[1]{\langle #1 \rangle}

\newcommand{\br}{\mathbb{R}}

\newcommand{\cd}{\mathcal{D}}

\newcommand{\cf}{\mathcal{F}}

\newcommand{\bn}{\mathbb{N}}
\newcommand{\be}{\mathbb{E}}
\newcommand{\bz}{\mathbb{Z}}
\newcommand{\ca}{\mathcal{A}}
\newcommand{\cs}{\mathcal{S}}

\newcommand{\Q}{|Q|}

\newcommand{\abs}[1]{\lvert#1\rvert}

\theoremstyle{plain}
\newtheorem{theorem}{Theorem}[section]
\newtheorem{corollary}[theorem]{Corollary}
\newtheorem{proposition}[theorem]{Proposition}

\newtheorem{question}[theorem]{Question}

\theoremstyle{definition}
\newtheorem{definition}[theorem]{Definition}

\newtheorem{lemma}[theorem]{Lemma}

\newtheorem{example}[theorem]{Example}

\theoremstyle{definition}

\theoremstyle{remark}
\newtheorem*{remark}{Remark}

\numberwithin{equation}{section}

\newcommand{\norm}[1]{\lVert#1\rVert}

\newcommand{\ch}{\textup{ch}}
\newcommand{\chf}{\textup{ch}_\mathcal{F}}
\newcommand{\chs}{\textup{ch}_\mathcal{S}}

\newcommand{\chg}{\textup{ch}_\mathcal{G}}

\newcommand{\eg}{E_{\mathcal{G}}}

\newcommand{\ef}{E_{\mathcal{F}}}
\newcommand{\es}{E_{\mathcal{S}}}

\newcommand{\pif}{{\pi_\mathcal{F}}}
\newcommand{\pis}{{\pi_\mathcal{S}}}
\newcommand{\pig}{{\pi_\mathcal{G}}}

\newcommand{\avo}[1]{\langle #1\rangle^\omega}

\newcommand{\avs}[1]{\langle #1\rangle^\sigma}

\newcommand{\cg}{\mathcal{G}}
\newcommand{\cah}{\mathcal{H}}

\begin{document}

\date{\today}
\subjclass[2010]{42B20,42B25,46E40}
     
\keywords{vector-valued,weight,testing conditions,norm inequality,parallel stopping cubes}

\title[Two-weight inequality for vector-valued positive dyadic operators]{Two weight inequality for vector-valued positive dyadic operators by parallel stopping cubes}

\author{Timo S. H\"anninen}
\address{Department of Mathematics and Statistics, University of Helsinki, P.O. Box 68, FI-00014 HELSINKI, FINLAND}
\email{timo.s.hanninen@helsinki.fi}

\begin{abstract}We study the vector-valued positive dyadic operator
\[T_\lambda(f\sigma):=\sum_{Q\in\mathcal{D}} \lambda_Q \int_Q f \,\mathrm{d}\sigma 1_Q,\]
where the coefficients $\{\lambda_Q:C\to D\}_{Q\in\mathcal{D}}$ are positive operators from a Banach lattice $C$ to a Banach lattice $D$. We assume that the Banach lattices $C$ and $D^*$ each have the Hardy--Littlewood property. An example of a Banach lattice with the Hardy--Littlewood property is a Lebesgue space.

In the two-weight case, we prove that the  $L^p_C(\sigma)\to L^q_D(\omega)$ boundedness of the operator $T_\lambda(\,\cdot\,\sigma)$ is characterized by the direct and the dual $L^\infty$ testing conditions:
\[
\lVert 1_Q T_\lambda(1_Q f \sigma)\rVert_{L^q_D(\omega)}\lesssim \lVert f\rVert_{L^\infty_C(Q,\sigma)} \sigma(Q)^{1/p},\]
\[
\lVert1_Q T^*_{\lambda}(1_Q g \omega)\rVert_{L^{p'}_{C^*}(\sigma)}\lesssim \lVert g\rVert_{L^\infty_{D^*}(Q,\omega)} \omega(Q)^{1/q'}.\] 
Here $L^p_C(\sigma)$ and $L^q_D(\omega)$ denote the Lebesgue--Bochner spaces associated with exponents $1<p\leq q<\infty$, and locally finite Borel measures $\sigma$ and $\omega$.

In the unweighted case, we show that the $L^p_C(\mu)\to L^p_D(\mu)$  boundedness of the operator $T_\lambda(\,\cdot\,\mu)$ is equivalent to the endpoint direct $L^\infty$ testing condition:
\[
\lVert1_Q T_\lambda(1_Q f \mu)\rVert_{L^1_D(\mu)}\lesssim \lVert f\rVert_{L^\infty_C(Q,\mu)} \mu(Q).\]
This condition is manifestly independent of the exponent $p$. By specializing this to particular cases, we recover some earlier results in a unified way.

\end{abstract}

\maketitle
\tableofcontents

\section*{Notation}
\begin{tabular*}{\textwidth}{c l}
$E$ & A Banach lattice $(E,\abs{\cdotroomy}_E,\leq)$.\\
$E_+$ & The positive cone of a Banach lattice, $E_+:=\{e\in E : e\geq 0\}$.\\
$E^*$ & The dual space of a Banach lattice, equipped with \\
& the order: $e^*\geq 0$ if and only if $e^*e\geq 0$ for all $e\in E_+$.\\
$\cd$ & A finite collection of dyadic cubes.\\
$\mu$ & A locally finite Borel measure.\\
$\dx$ & The Lebesgue measure.\\
$\abs{Q}$ & The Lebesgue measure of a set $Q$.\\
$\angles{f}^\mu_Q$ & The average $\angles{f}^\mu_Q:=\frac{1}{\mu(Q)}\int_Q f \dmu$.\\
$\angles{f}_Q$ & The average $\angles{f}_Q:=\angles{f}^{\dx}_Q$.\\
$L^p_E(\mu)$ & The Lebesgue--Bochner space, \\
&equipped with the norm $\norm{f}_{L^p_E(\mu)}:=(\int \abs{f}_E^p \dmu)^{1/p}.$\\
$L^p_E$ & The Lebesgue--Bochner space $L^p_E:=L^p_E(\dx)$.\\
$\mbar^\mu_{\cd}$ & The lattice maximal function: $\mbar^\mu f:=\sup_{Q\in\cd} \angles{f}_Q^\mu 1_Q$,\\
& where the supremum is taken in the lattice order.\\
$\norm{\mbar^\mu}_{L^p_E(\mu)\to L^p_E(\mu)}$ & Shorthand for the uniform bound:  $\sup_{\text{$\cd$}} \norm{\mbar^\mu_\cd}_{L^p_E(\mu)\to L^p_E(\mu)}. $
\end{tabular*}
\section{Introduction and the main results}
Let $(C,\abs{\cdotroomy}_C,\leq )$ and $(D, \abs{\cdotroomy}_D,\leq )$ be Banach lattices. We consider the vector-valued positive dyadic operator $T_\lambda(\, \cdot\, \sigma)$ defined as follows: For every locally integrable function $f:\br^d\to C$, the function $T_\lambda(f\sigma):\br^d\to D$ is defined by
\begin{equation}\label{eq_operator}
T_\lambda(f\sigma):=\sum_{Q\in\cd} \lambda_Q \int_Q f \dsigma 1_Q,
\end{equation}
where $\cd$ is a finite collection of dyadic cubes on $\br^d$, $\sigma$ is a locally finite Borel measure, and $\{ \lambda_Q :C\to D \}_{Q\in\cd}$ are positive operators. 

Let $L^p_C(\sigma)$ and $L^q_D(\omega)$ denote the Lebesgue--Bochner spaces associated with the exponents $1<p\leq q<\infty$, locally finite Borel measures $\sigma$ and $\omega$, and the Banach lattices $C$ and $D$. We assume that $C$ and $D^*$ each have the Hardy--Littlewood property. We characterize the two-weight norm inequality
\begin{equation}\label{eq_norminequality}
\norm{T_\lambda(f\sigma)}_{\llq} \lesssim \norm{f}_{\lp}
\end{equation}
by means of testing conditions.  
Furthermore, we characterize the unweighted norm inequality
\begin{equation*}\label{eq_unweightednorminequality}
\norm{T_\lambda(f\mu)}_{L^q_D(\mu)} \lesssim \norm{f}_{L^p_C(\mu)}
\end{equation*} 
by means of an end-point testing condition. Among the corollaries of this characterization is that the operator $T_\lambda(\cdotroomy \mu):L^p_C(\mu) \to L^p_D(\mu)$ is bounded for some $p\in(1,\infty)$ if and only if it is bounded for every $p\in(1,\infty)$. 

A {\it Banach lattice} $(C,\abs{\cdotroomy}_C,\leq)$ is a Banach space $(C,\abs{\cdotroomy}_C)$ equipped with a partial order $\leq$ that is compatible with the vector addition, the scalar multiplication, and the norm of the Banach space, and such that each pair of vectors has the least upper bound, or, in other words, the supremum. (The precise definition of a Banach lattice is given in Section \ref{subsec_rudimentsbanachlattices}.) A linear operator $\lambda:C\to D$ from a Banach lattice $C$ to a Banach lattice $D$ is {\it positive} if $c\geq 0$ implies $Tc\geq 0$, for every $c\in C$. The {\it dyadic lattice  Hardy--Littlewood maximal operator} $\bar{M}_\cd:L^p_C\to L^p_C$ is defined by
\begin{equation}\label{eq_latticemaximalfcn}
\bar{M}_\cd f:=\sup_{Q\in\cd} \angles{f}_Q 1_Q,
\end{equation}
where the supremum is taken with respect to the order of the lattice. 
\begin{definition}[Dyadic Hardy--Littlewood property]\label{def_hardylittlewoodproperty}A Banach lattice $(E, \abs{\cdotroomy}_E,\leq)$  has the {\it dyadic Hardy--Littlewood property} if for some $p\in(1,\infty)$ there exists a finite constant $C_{p,E}$ such that 
\begin{equation}\label{eq_hardylittlewoodproperty}
\norm{\mbar_\cd}_{L^p_E\to L^p_E} \leq C_{p,E} 
\end{equation}
for every finite collection $\cd$ of dyadic cubes.
\end{definition}
\begin{remark}
The estimate \eqref{eq_hardylittlewoodproperty} holds for some $p\in(1,\infty)$ if and only if it holds for every $p\in(1,\infty)$, as proven by Garc\'ia-Cuerva, Mac\'ias, and Torrea in \cite{torrea1993}.
\end{remark}

\begin{example}

a) The Lebesgue space $L^r(A,\ca,\alpha)$ associated with an exponent $r\in(1,\infty)$ and a $\sigma$-finite measure space $(A,\ca,\alpha)$ is a Banach lattice that has the dyadic Hardy--Littlewood property, which is a choice of words for saying that the dyadic Fefferman--Stein vector-valued maximal inequality \cite{fefferman1971} holds:
$$
\norm{\bar{M} }_{L^p_{L^r(A)}\to L^p_{L^r(A)}} \leq C_{p,r}.
$$

\noindent b) A K\"othe function space $X$ with the Fatou property has the UMD property if and only if both $X$ and its function space dual $X'$ have the Hardy--Littlewood property, as proven by Bourgain,  and Rubio de Francia (see \cite{bourgain1984}, and \cite{rubiodefrancia1985}).
\end{example}
The Hardy--Littlewood property is studied by Garc\'ia-Cuerva, Mac\'ias, and Torrea in \cite{torrea1993} and \cite{torrea1998}. Among other things, they obtain various characterizations of the property. 
In fact, they define the Hardy--Littlewood property by means of the Hardy--Littlewood maximal operator with the supremum taken over centered balls,
whereas we define it with the supremum taken over dyadic cubes. In any case, for the Lebesgue measure, these maximal functions are comparable, as explained in Section \ref{sec_comparisionofdyadicandcentered}.

By duality, the norm inequality \eqref{eq_norminequality} for the operator $T_\lambda(\cdotroomy \sigma):\lp\to \llq$ is equivalent to the norm inequality
\begin{equation}\label{eq_dualnorminequality}
\norm{T^*_\lambda(g\omega)}_{\lpp}\lesssim \norm{g}_{\lqp}
\end{equation}
for the adjoint operator $T^*_\lambda(\cdotroomy \omega):\lqp\to\lpp$  defined by
\begin{equation*}\label{eq_adjointoperator}
T^*_\lambda(g\omega):=\sum_{Q\in\cd} \lambda_Q^* \int_Q g \domega 1_Q.
\end{equation*}
The localized versions $T_R$ of the operator $T$ and the localized version $T^*_R$ of its adjoint $T^*$ are defined by
\begin{equation}\label{eq_localization}
T_{\lambda,R}(f\sigma):=\sum_{\substack{Q\in\cd:\\Q\subseteq R}} \lambda_Q \int_Q f \dsigma 1_Q\quad \text{ and } \quad T^*_{\lambda,R}(g\omega):=\sum_{\substack{Q\in\cd:\\Q\subseteq R}} \lambda_Q^* \int_Q g \domega 1_Q.
\end{equation}

The characterization of the norm inequality \eqref{eq_norminequality} is obtained by weakening it and its dual \eqref{eq_dualnorminequality} by restricting the class of functions and by localizing the operator $T$ and its adjoint $T^*$ as in \eqref{eq_localization}. Thus, we obtain {\it the direct and the dual $L^\infty$ testing condition}:
\begin{subequations}
\label{eq_linftytestingconditions}
\begin{align}
\label{eq_testing_direct}\norm{T_R(f\sigma)}_{\llq}&\leq \mathfrak{T} \norm{f}_{L^\infty_C(R,\sigma)} \sigma(R)^{1/p}, \\
\label{eq_testing_dual}\norm{T^*_R(g\,\omega)}_{\lpp}&\leq \mathfrak{T}^* \norm{g}_{L^\infty_{D^*}(R,\omega)} \omega(R)^{1/q'},
\end{align}
\end{subequations}
for every $R\in\cd$, every $f\in L^\infty_C(R,\sigma)$, and every $g\in L^\infty_{D^*}(\omega,R)$.

\begin{theorem}[Two-weight norm inequality is characterized by the direct and the dual $L^\infty$ testing conditions] \label{thm_twoweight}Let $1<p\leq q<\infty$. 
Let $\sigma$ and $\omega$ be locally finite Borel measures. 
Let $C$ and $D$ be Banach lattices. Assume that $C$ and $D^*$ each have the dyadic Hardy--Littlewood property. Let $\{\lambda_Q:C\to D\}_{Q\in\cd}$ be positive operators. Let the operator $T_\lambda(\cdotroomy \sigma)$ be defined as in \eqref{eq_operator}, and the localizations $T_{\lambda,R}(\cdotroomy \sigma)$ and $T^*_{\lambda,R}(\cdotroomy \omega)$ as in \eqref{eq_localization}.  Then, 
$$ 
\max\{\mathfrak{T},\mathfrak{T}^*\}\leq \norm{T(\,\cdot\,\sigma)}_{L^p_C(\sigma)\to L^q_D(\omega)} \lesssim_{q,p} \norm{\bar{M}}_{L^p_C\to L^p_C} \mathfrak{T} +  \norm{\bar{M}}_{L^{q'}_{D^*}\to L^{q'}_{D^*}} \mathfrak{T^*},
$$
where the testing constants $\mathfrak{T}$ and $\mathfrak{T^*}$ are the least constants in the testing conditions \eqref{eq_testing_direct} and \eqref{eq_testing_dual}.
Here, $\norm{\bar{M}}_{L^p_C\to L^p_C}$ denotes the norm of the dyadic lattice  Hardy--Littlewood maximal operator $\bar{M}:L^p_C\to L^p_C$ defined in $\eqref{eq_latticemaximalfcn}$. 
\end{theorem}
We note that, in the real-valued case (that is, $C=D=\br$), the $L^\infty$ testing conditions \eqref{eq_linftytestingconditions} can be rephrased as the Sawyer testing conditions:
\begin{equation}\label{eq_sawyer testing}
\norm{T_R(1_R\sigma)}_{L^q(\omega)}\lesssim \sigma(R)^{1/p},\quad \text{ and }\quad \norm{T^*_R(1_R\omega)}_{L^{p'}(\sigma)}\lesssim \omega(R)^{1/{q'}}.
\end{equation}
Such testing conditions were used by Sawyer \cite{sawyer1988} to characterize the boundedness of a large class of integral operators $I(\,\cdot\,\sigma):L^p(\sigma)\to L^q(\omega)$ with non-negative kernels, in particular, fractional integrals and Poisson integrals. 
In the real-valued case $T(\cdotroomy\sigma):L^p(\sigma)\to L^q(\omega)$, Theorem \ref{thm_twoweight} was first proven 
\begin{itemize}
\item for $p=q=2$ by Nazarov, Treil, and Volberg  \cite{nazarov1999} by the Bellman function technique,  
\item and for $1<p\leq q<\infty$ by Lacey, Sawyer, and Uriarte-Tuero \cite{lacey2009} by techniques that are similar to the ones used by Sawyer \cite{sawyer1988};
\end{itemize}
Alternative proofs were obtained
\begin{itemize}
\item by Treil \cite{treil2012} by splitting the summation over dyadic cubes in the dual pairing by the condition \lq$\sigma(Q)(\avs{f}_Q)^p>\omega(Q)(\avo{g}_Q)^{q'}$\rq, 
\item and by  
Hyt\"onen \cite{hytonen2012} by splitting the summation by using parallel stopping cubes. This technique originates from the work of Lacey, Sawyer, Shen, and Uriarte-Tuero \cite[Version 1]{lacey2012} on the two-weight boundedness of the Hilbert transform.  
\end{itemize}
For an exponent $s\in(1,\infty)$, and a collection $\{\beta_Q\}_{Q\in\cd}$ of non-negative real numbers, consider the particular vector-valued case $L^p(\sigma)\to L^q_{\ell^s(\cd)}(\omega)$, and the particular class of operators $T_{\lambda_\beta}(\cdotroomy \sigma)$ defined by
\begin{equation}\label{eq_operator_particular}
T_{\lambda_\beta}(f\sigma):=\{ \beta_Q \int_Q f \dsigma 1_Q \}_{Q\in\cd}.
\end{equation}
(We note that this is the operator \eqref{eq_operator} associated with the following coefficients: For each $Q\in\cd$, for every $r\in\br$, the sequence $\lambda_{\beta,Q}r\in\ell^s(\cd)$ is componentwise defined by setting $(\lambda_{\beta,Q} r)_{R}:=\delta_{Q,R} \beta_Q r$ for every $ R\in\cd$.) In this case, Theorem \ref{thm_twoweight} was proven
\begin{itemize}
\item by Scurry \cite{scurry2010} by adapting Lacey, Sawyer, and Uriarte-Tuero's {\cite{lacey2009}} proof of the real-valued case $T(\cdotroomy\sigma):L^p(\sigma)\to L^q(\omega)$.
\end{itemize}
In this paper, the characterization by the $L^\infty$ testing conditions  is extended to Banach lattices with the Hardy--Littlewood property. Note that this generality also has the advantage of being symmetric with respect to $T$ and $T^*$, which simplifies the notation.

We prove Theorem \ref{thm_twoweight} by using parallel stopping cubes, similarly as in Hyt\"onen's \cite{hytonen2012} proof of the real-valued case $L^p(\sigma)\to L^q(\omega)$ of the theorem. However, because of the vector-valuedness, we need to choose the stopping cubes  by a different stopping condition: Let $\mu$ be a locally finite Borel measure, and let  $(E,\abs{\cdotroomy}_E,\leq)$ be a Banach lattice. For each dyadic cube $F$, its stopping children $\chf(F)$ are defined as the maximal dyadic cubes $F'\subsetneq F$ such that
\begin{equation}\label{eq_strongstoppingcondition}
\abs{\sup_{\substack{Q\in\cd:\\Q\supseteq F'}}\angles{f}^\mu_Q}_E>2 \angles{\abs{ \sup_{Q\in\cd}\angles{f}^\mu_Q 1_Q}_E}^\mu_F,
\end{equation}
where the supremum is taken with respect to the order of the lattice. 

Note that, in the right-hand side of the stopping condition \eqref{eq_strongstoppingcondition}, there appears the  dyadic lattice Hardy--Littlewood maximal function $
\bar{M}^\mu_\cd f$, which is defined by $\bar{M}^\mu_\cd f:=\sup_{Q\in\cd} \angles{f}^\mu_Q 1_Q.
$
To control the averages appearing in the stopping condition \eqref{eq_strongstoppingcondition}, we assume that the operator $\bar{M}^\mu: L^p_E(\mu)\to L^p_E(\mu)$ is bounded. However, we want to obtain an estimate for the operator norm of the operator $T(\cdotroomy \sigma):L^p_C(\sigma)\to L^q_D(\omega)$ such that the estimate depends on the measures $\sigma$ and $\omega$ only via the testing contants. 
In particular, we do not want the estimate to depend on the measure $\sigma$ via the operator norm of the auxiliary operator $\bar{M}^\sigma: L^p_C(\sigma)\to L^p_C(\sigma)$. 
Thus, we want to view the boundedness of $\bar{M}^\sigma: L^p_E(\sigma)\to L^p_E(\sigma)$ as a consequence of the geometry of the Banach lattice $E$ itself, which we can do, thanks to the following theorem:

\begin{theorem}[Universal norm bound for the dyadic lattice Hardy--Littlewood maximal operator, \cite{nazarov1996} and \cite{kemppainen2011}]\label{thm_universalbound}Let $1<p<\infty$. Assume that $(E, \abs{\cdotroomy}_E,\leq )$ is a Banach lattice. Then
\begin{equation*}
\norm{\bar{M}^\mu}_{L^p_E(\mu)\to L^p_E(\mu)}\lesssim_p \norm{\bar{M}}_{L^p_E\to L^p_E}
\end{equation*}
for all locally finite Borel measures $\mu$. 
\end{theorem}

\begin{remark}This theorem  follows from either the technique \cite{nazarov1996} or, as communicated to the author by M. Kemppainen, the technique \cite{kemppainen2011}. For reader's convenience, the proof is presented in Section \ref{sec_universalnormbound}. 
\end{remark}

Thus, it is the proof technique of stopping cubes, in particular, the stopping condition \eqref{eq_strongstoppingcondition}, that leads us to consider the class of Banach lattices that have the Hardy--Littlewood property. The author is unaware of whether the statement, the characterization of the two-weight boundedness by the $L^\infty$ testing conditions, holds  without assuming the Hardy--Littlewood property (see Question \ref{question_additionalassumption}).

Next, we characterize the $L^p_C(\sigma)\to L^q_D(\omega)$ boundedness of the operator $T_\lambda(\cdotroomy\sigma)$ in the case that the measures $\sigma$ and $\omega$ satisfy the $A_\infty$ condition with respect to each other. In particular, this includes the unweighted case $\sigma=\omega=\mu$. By duality, the norm inequality \eqref{eq_norminequality} is equivalent to the bilinear norm inequality
\begin{equation}\label{eq_dualpairingnorminequality}
\int g T(f\sigma) \domega \lesssim \norm{f}_{\lp} \norm{g}_{\lqp}.
\end{equation}
Again, by restricting the class of functions and by localizing the operator, we obtain {\it the $L^\infty$ dual pairing testing condition}:
\begin{equation}\label{eq_dualpairingtestingcondition}
\int g T_R(f\sigma) \domega \leq \mathfrak{B}\norm{f}_{L^\infty_C(R,\sigma)} \norm{g}_{L^\infty_{D^*}(R,\omega)}\sigma(R)^{1/p} \omega(R)^{1/{q'}}
\end{equation}
for every $R\in\cd$, every $g\in L^\infty_{D^*}(\omega,R)$, and every $f\in L^\infty_C(\sigma,R)$.
The  {\it $A_\infty$ characteristic} $[\sigma]_{A_\infty(\omega)}$ of a measure $\sigma$ with respect to a measure $\omega$ is defined by
\begin{equation}\label{eq_ainfinitycharacteristic}
[\sigma]_{A_\infty(\omega)}:= \sup_{R\in\cd} \frac{1}{\sigma(R)} \int M_R^\omega(\sigma) \domega,
\end{equation}
where, for each $R\in\cd$, the localized Hardy--Littlewood maximal operator $M^\omega_R$ is defined by
$
M^\omega_R(\sigma):=\sup_{\substack{Q\in\cd:\\Q\subseteq R}} \frac{\sigma(Q)}{\omega(Q)} 1_Q.
$

\begin{theorem}[Norm inequality for $A_\infty$ weights is characterized by the $L^\infty$ dual pairing testing condition]\label{thm_dualpairingtesting}

In addition to the assumptions of Theorem \ref{thm_twoweight}, assume that the measures $\sigma$ and $\omega$ satisfy the $A_\infty$ condition with respect to each other. 
Then 
$$
\mathfrak{B}\leq \norm{T_\lambda(\,\cdot\,\sigma)}_{L^p_C(\sigma)\to L^q_D(\omega)} \lesssim_{p,q} \norm{\bar{M}}_{L^p_C\to L^p_C} \norm{\bar{M}}_{L^{q'}_{D^*}\to L^{q'}_{D^*}} \big([\sigma]_{A_\infty(\omega)}^{1/p} +[\omega]_{A_\infty(\sigma)}^{1/{q'}}\big) \mathfrak{B},
$$
where the dual pairing testing constant $\mathfrak{B}$ is the least constant in the dual pairing testing condition \eqref{eq_dualpairingtestingcondition}.
Here, the $A_\infty$ characteristics are defined as in \eqref{eq_ainfinitycharacteristic},
and $\norm{\bar{M}}_{L^p_C\to L^p_C}$ denotes the norm of the dyadic lattice Hardy--Littlewood maximal function $\bar{M}:L^p_C\to L^p_C$.
\end{theorem}

We observe that the $L^\infty$ dual pairing testing condition \eqref{eq_dualpairingtestingcondition} for $T_\lambda(\cdotroomy \mu):L^p_C(\mu)\to L^p_D(\mu)$ is independent of $p$. Therefore:
\begin{corollary}Assume that $C$ and $D^*$ each have the Hardy--Littlewood property. Then, the operator $T_\lambda(\cdotroomy \mu):L^p_C(\mu)\to L^p_D(\mu)$ is bounded for some $p\in(1,\infty)$ if and only if it is bounded for every $p\in(1,\infty)$.
\end{corollary}
More corollaries, among which is is an alternative proof for an embedding theorem by Nazarov, Treil, and Volberg \cite[Theorem 3.1]{nazarov2003}, are stated in Section \ref{sec_corollaries}.

Next, we point out that the assumption that the Banach space has the Hardy--Littlewood property can be replaced by assuming that the measure is doubling, or by strenghtening the testing condition. In the unweighted case $T_\lambda(\cdotroomy \mu):L^p_E(\mu)\to L^p_E(\mu)$, this reads as:
\begin{theorem}[$L^\infty$ testing condition together with an additional assumption implies the boundedness]\label{thm_alternativetesting}Let $p\in(1,\infty)$. Let $(E,\abs{\cdotroomy}_E,\leq)$ be a Banach lattice. Let $\mu$ be a locally finite Borel measure. Then, the operator $T_\lambda(\cdotroomy \mu):L^p_E(\mu)\to L^p_E(\mu)$ is bounded if any of the following conditions is satisfied:
\begin{itemize}
\item[i)]The operator $T_\lambda(\cdotroomy)$ satisfies the endpoint direct $L^\infty$ testing condition:
\begin{equation}
\label{eq_testing_endpointdirect}
\norm{T_R(f\mu)}_{L^1_E(\mu)}\leq \mathfrak{B} \norm{f}_{L^\infty_E(R,\mu)} \mu(R)
\end{equation}
for every $R\in\cd$, and every $f\in L^\infty_E(R,\mu)$, and, additionally, the Banach lattice $E$ has the Hardy--Littlewood property.
\item[ii)]The operator $T_\lambda(\cdotroomy)$ satisfies the endpoint direct $L^\infty$ testing condition \eqref{eq_testing_endpointdirect}, and, additionally, the measure $\mu$ is doubling.
\item[iii)]The operator $T_\lambda(\cdotroomy)$ satisfies,  for some  $t\in(p,\infty)$, the endpoint direct $L^t$ testing condition: 
\begin{equation}\label{eq_ltendpointtestingcondition}
\norm{T_R(f\mu)}_{L^1_E(\mu)} \leq \mathfrak{B}_{t} \norm{f}_{L^t_E(\mu,R)} \mu(R)^{1-1/t}
\end{equation} 
for every $R\in\cd$ and every $f\in L^t_E(R,\mu)$. 
\end{itemize}
\end{theorem}

We remark that the $L^\infty$ testing condition 
has been used to characterize $L^p_E\to L^p_E$ boundedness in at least the following instances: 
\begin{itemize}
\item Let $(E,\abs{\cdotroomy}_E,\leq )$ be a Banach lattice. By using the theory of vector-valued singular integrals, Garc\'ia-Cuerva, Mac\'ias, and Torrea  \cite{torrea1993} proved that {\it the smooth lattice Hardy--Littlewood maximal operator} $\bar{M}_{\varphi,J}:L^p_E\to L^p_E$ is bounded if and only if it satisfies the end-point direct $L^\infty$ testing condition \eqref{eq_testing_endpointdirect}. An alternative proof for this is given in Section \ref{sec_alternativeproofformaximal} by using stopping cubes.
 \item Let $(E,\abs{\cdotroomy}_E)$ be a UMD space. By using stopping cubes, the author and Hyt\"onen \cite{hanninenhytonen2014} proved that {\it the operator-valued dyadic paraproduct} $\Pi_b:L^p_E\to L^p_E$ is bounded if and only if it satisfies the direct $L^\infty$ testing condition \eqref{eq_testing_direct}.
\end{itemize}

We conclude the introduction by comparing the testing conditions. Observe that the direct $L^\infty$ testing condition \eqref{eq_testing_direct} or the dual $L^\infty$ testing condition \eqref{eq_testing_dual} each imply, by H\"older's inequality, the $L^\infty$ dual pairing testing condition \eqref{eq_dualpairingtestingcondition}. Furthermore, the direct $L^t$ testing condition,
\begin{equation}\label{eq_lttestingcondition}
\norm{T_R(f\sigma)}_{L^q_D (\omega)} \leq \mathfrak{T}_{t} \norm{f}_{L^t_C(\sigma,R)} \sigma(R)^{1/p-1/t}
\end{equation}
for every $R\in\cd$, and every $f\in L^t_C(\sigma,R)$, implies, again by H\"older's inequality, the direct $L^\infty$ testing condition \eqref{eq_testing_direct}. Altogether, the testing constants satisfy the comparision:
$$\mathfrak{B}\leq \mathfrak{T}\leq \mathfrak{T}_t\leq \norm{T(\cdotroomy \sigma)}_{L^p(\sigma)\to L^q(\omega)}.$$ 
The $L^\infty$ testing condition \eqref{eq_testing_direct}
can be viewed as the limiting case ($t=\infty)$ of the $L^t$ testing condition \eqref{eq_lttestingcondition}. Furthermore, the $L^\infty$ dual pairing testing condition \eqref{eq_dualpairingtestingcondition} 
is, by duality,  equivalent to {\it the end-point direct $L^\infty$ condition} or {\it the end-point dual $L^\infty$ condition}:
\begin{subequations}
\label{eq_limitingcaseoflinfinity}
\begin{align}
\label{eq_limitingcaseoflinfinitydirect}\norm{T_R(f\sigma)}_{L^1_D(\omega)} &\lesssim \norm{f}_{L^\infty_C(R,\sigma)}\sigma(R)^{1/p}\omega(R)^{1/q'},\;\; \\
\norm{T^*_R(g\omega)}_{L^1_{C^*}(\sigma)}&\lesssim \norm{g}_{L^\infty_{D^*}(R,\omega)} \sigma(R)^{1/p}\omega(R)^{1/q'}.
\end{align}
\end{subequations}
In particular, in the unweighted case $T(\cdotroomy \mu):L^p_C(\mu)\to L^p_D(\mu)$, these conditions can be viewed as the limiting case of the $L^\infty$ testing conditions \eqref{eq_linftytestingconditions}.

\section{Preliminaries}

\subsection{Rudiments of Banach lattices}\label{subsec_rudimentsbanachlattices}
A {\it lattice } $(C,\leq)$ is a set equipped with a partial order relation $\leq$ such that for every $c,d\in C$ there exists the least upper bound $c\vee d $ and the greatest lower bound $c\wedge d$.

\begin{definition}[Banach lattice]
A {\it Banach lattice} $(C,\abs{\cdotroomy}_C,\leq)$ is both a real Banach space $(C,\abs{\cdotroomy}_C)$ and a lattice $(C,\leq)$ so that both structures are compatible: 
\begin{itemize}
\item[i)]$c\leq d$ implies
$c+e\leq d+e$, for every $c,d,e\in C$. 
\item[ii)]$r\geq 0$ and $c\geq 0$ implies $rc\geq 0$, for every  $r\in\br$ and $c\in C$.  
\item[iii)]$\abs{c\,}_C=\abs{\,\abs{c}\,}_C$, and $0\leq c\leq d$ implies $\abs{c\,}_C\leq \abs{d\,}_C$, for every $c,d\in C$. Here, 
the {\it positive part} $c_+$ of a vector $c\in C$ is defined by $c_+:=c\vee 0$, the {\it negative part} $c_-$ by $c_-:=-c\vee 0 $, and the {\it  absolute value} $\abs{c}$ by $\abs{c}:=c\vee -c$. 
\end{itemize}
\end{definition}
From the existence of the pairwise supremum (in other words, the least upper bound), it follows that for every finite set there exists the supremum. This supremum can be computed by taking pairwise suprema and using the recursive formula $\sup \{c_n\}_{n=1}^N=\sup\{c_n\}_{n=1}^{N-1} \vee c_N$.

From the definitions, it follows that $c=c_+-c_-$, and $\abs{c}=c_++c_-$ for every $c\in C$. This splitting implies that, for every linear operator $T:C\to D$ from a Banach lattice $C$ to another $D$, the norm estimate $\abs{Tc\,}_D\lesssim \abs{c}_C$ holds for all $c\in C$ if and only if it holds for all $c\in C$ such that $c\geq 0$. 

The Lebesgue--Bochner space $L^p_C(\sigma)$ associated with a Banach lattice $(C,\abs{\cdotroomy}_C,\leq)$ is again a Banach lattice. The order is defined by using  the lattice order pointwise: For $f_1,f_2\in L^p_C(\sigma)$, we impose that $f_1\leq f_2$ if and only $f_1(x)\leq f_2(x)$ for $\sigma$-almost every $x\in\br^d$. 

\subsubsection*{Dual of a Banach lattice}
The dual $C^*$ of a Banach lattice $C$ is also a Banach lattice, provided that it is equipped with the lattice order defined as follows: For $c^*,d^*\in C^*$, we impose
\begin{equation}
\label{eq_orderofthedual}
c^*\leq  d^* \text{ if and only if } c^*c\leq  d^*c \text{ for every $c\in C$ with $c\geq 0$}.
\end{equation}

In this paper, it is implicitly understood that the dual of a Banach lattice is equipped with this lattice order. The supremum $c^*\vee d^*$ of $c^*,d^*\in C^*$ is given by
$$
(c^*\vee d^*)(c)=\sup \{c^*(d)+d^*(c-d) : 0\leq d \leq c \}.
$$
\subsubsection*{Positive operator}
An operator $T:C\to D$ from a Banach lattice $C$ to a Banach lattice $D$ is {\it positive} if $c\geq 0$ implies $Tc\geq 0$, for every $c\in C$. By the definition of the lattice order of the dual \eqref{eq_orderofthedual}, the adjoint $T^*:D^*\to C^*$ of a positive operator $T:C\to D$ is also a positive operator, which reads
\begin{equation*}
(T^*d^*) c=d^* (Tc)\geq 0\text{ for every $d^*\in D^*$ with $d^*\geq 0$ and $c\in C$ with $c\geq 0$}.
\end{equation*}

For more on Banach lattices, see Lindenstrauss and Tzafriri's book \cite[Chapter 1]{lindenstrauss1979}.

\subsection{Stopping families and dyadic analysis}
\subsubsection{Terminology}
Let $\cs$ be a collection of dyadic cubes. Let $\mu$ be a locally finite Borel measure.
\begin{itemize}\item {\it $\cs$-children} of $S\in\cs$, denoted by $\chs(S)$, are defined by
$$
\chs(S):=\{S'\in\cs : \text{$S'$ maximal with $S'\subsetneq S$}\}.
$$
\item  {\it $\cs$-parent} of $Q\in\cd$, denoted by $\pis(Q)$, is defined by
$$
\pis(Q):=\{S\in\cs : \text{$S$ minimal with $S\supseteq Q$}\}.
$$
\item $ \es(S):=S\setminus \bigcup_{S'\in\chs(S)} S'$.
\item Let $0<c<1$. The collection $\cs$ is {\it $(c,\mu)$-sparse} if, for every $S\in\cs$,
\begin{equation}\label{def_strictsparseness}
\mu(\es(S))\geq c\mu(S).
\end{equation} By taking the complement, this is equivalent to the condition that, for every $S\in\cs$, $$\sum_{S'\in\cs(S)}\mu(S')\leq (1-c)\mu(S).$$
In the case that the constant $c$ is not explicitly specified, we use the convention that $c=\frac{1}{2}$.
\item Let $C>1$. The collection $\cs$ is {\it $(C,\mu)$-Carleson} if, for every $S\in\cs$, $$\sum_{\substack{S'\in\cs:\\S'\subseteq S}} \mu(S')\leq C \mu(S).$$ 
In the case that the constant $C$ is not explicitly specified, we use the convention that $C=2$.
\item For each $Q\in\cd$, let  $\chs(Q)$ be a collection of pairwise disjoint dyadic subcubes of $Q$. We say that {\it $\cs$ is the family starting at a dyadic cube $S_0$ and defined by the children $\chs$} if $\cs$ is defined recursively as follows: $\cs_0:=\{S_0\}$, $\cs_{k+1}:=\bigcup_{S\in\cs_k} \chs(S)$, and $\cs:=\bigcup_{k=0}^\infty \cs_k$. (Once $\cs$ is defined so, then $\chs(S)=\{S'\in\cs : \text{$S'$ maximal with $S'\subsetneq S$}\}$, for every $S\in\cs$.)
\end{itemize}
\subsubsection{Basic lemmas}
The {\it dyadic (real-valued) Hardy--Littlewood maximal operator} $M^\mu$ is defined by
$$
M^\mu h:= \sup_{Q\in\cd} \angles{h}^\mu_Q 1_Q.
$$
\begin{lemma}[Universal norm bound for the dyadic Hardy--Littlewood maximal operator]Let $1<p\leq \infty$. Let $\mu$ be a locally finite Borel measure. Then
$$
\norm{M^\mu}_{L^p(\mu)\to L^p(\mu)}\leq p'.
$$
\end{lemma}

\begin{lemma}[Dyadic Carleson embedding theorem]\label{lem_carlesonembedding}Let $1<p<\infty$. Let $\mu$ be a locally finite Borel measure. Let $E$ be a Banach space. Suppose that $\cs$ is a sparse  collection. Then
$$
\Big(\sum_{S\in\cs} \big(\angles{\abs{f}_E}^\mu_S \big)^p \mu(S) \Big)^{1/p}\leq 2 p' \norm{f}_{L^p_E(\mu)}.
$$
\end{lemma}

\begin{lemma}[$L^p$-variant of Pythagoras' theorem, Lemma 2.7 in \cite{hanninenhytonen2014}]\label{lem_pythagoras}Let $1\leq p <\infty$. Let $\mu$ be a locally finite Borel measure. Let $E$ be a Banach space. Assume that $\cs$ is a sparse collection of dyadic cubes. Assume that $\{f_S\}_{S\in\cs}$ is a collection of $E$-valued functions such that every $f_S$ is supported on $S$ and constant on each $S'\in\chs(S)$. Then
$$
\norm{\sum_{S\in\cs} f_S}_{L^p_E(\mu)}\leq 3p \Big(\sum_{S\in\cs} \norm{f_S}_{L^p_E(\mu)}^p\Big)^{1/p}.
$$ 
\end{lemma}

\subsection{Equivalence of the $A_\infty$ condition and the Carleson condition}
The equivalence presented in this section is well-known. However, for reader's convenience, we represent a proof for it.
\begin{lemma}[Equivalence of the $A_\infty$ condition and the Carleson condition]\label{lem_ainfinitycarleson}Let $\sigma$ and $\omega$ be locally finite Borel measures. Then the measure $\sigma$ satisfies the $A_\infty$ condition with respect to the measure $\omega$ if and only if every $\omega$-Carleson collection is also $\sigma$-Carleson. Quantitatively,
$$
[\sigma]_{A_\infty (\omega)}\eqsim [\sigma]_{\operatorname{Car}(\omega)},
$$
where
$$
[\sigma]_{A_\infty (\omega)}:=\sup_{Q\in\cd} \frac{1}{\sigma(Q)}\int M_Q^\omega(\sigma)\domega,\quad [ \sigma]_{\operatorname{Car}(\omega)}:=\sup_{\substack{\cg\subseteq \cd:\\ \text{$\cg$ $w$-Carleson}}} \sup_{G\in\cg} \frac{1}{\sigma(G)}\sum_{\substack{G'\in\cg:\\G'\subseteq G}}\sigma(G').
$$
\end{lemma}
\begin{proof}First, we prove that $[\sigma]_{\operatorname{Car}(\omega)}\lesssim [\sigma]_{A_\infty (\omega)}$. Let $\cah$ be an $\omega$-Carleson collection. Fix $H_0\in\cah$. Let $\cg$ be the stopping family starting at $H_0$ and defined by
$$\chg(G):=\{G'\in\cah : G'\subseteq G \text{ maximal with } \frac{\sigma(G')}{\omega(G')}>2 \frac{\sigma(G)}{\omega(G)}\}.$$
Observe that the collection $\cg$ is $\omega$-sparse because
$$
\sum_{G'\in\chg(G)}\omega(G')<\frac{1}{2} \omega(G) \Big( \frac{1}{\sigma(G)}\sum_{G'\in\chg(G)} \sigma(G')\Big)\leq \frac{1}{2}\omega(G).
$$
Let $\eg(G):=G\setminus \bigcup_{G'\in\chg(G)} G'$. Moreover, observe that $\pig(H)=G$ implies that $H$ satisfies the opposite of the stopping condition. Altogether,
\begin{itemize}
\item The sets $\eg(G)$ are pairwise disjoint and satisfy $\omega(G)\leq 2 \omega(\eg(G))$.
\item $\frac{\sigma(H)}{\omega(H)}\leq 2 \frac{\sigma(G)}{\omega(G)}$ whenever $G\in \cg$ and $H\in\cah$ are such that $\pig(H)=G$.
\end{itemize} 
Now,
\begin{equation*}
\begin{split}
&\sum_{\substack{H\in\cah:\\ H\subseteq H_0}} \sigma(H)= \sum_{G\in\cg} \sum_{\substack{H\in\cah:\\ \pig(H)=G}}\frac{\sigma(H)}{\omega(H)}\omega(H)\leq 2 \sum_{G\in\cg} \frac{\sigma(G)}{\omega(G)} \sum_{\substack{H\in\cah:\\ H \subseteq G}} \omega(H)\leq 4 \sum_{G\in\cg} \frac{\sigma(G)}{\omega(G)} \omega(G)\\
& \leq 8  \sum_{G\in\cg} \frac{\sigma(G)}{\omega(G)} \omega(\eg(G))\leq  8 \int_{H_0}  M^\omega_G(\sigma) \domega \leq 8 [\sigma]_{A_\infty(\omega)} \sigma(H_0).
\end{split}
\end{equation*}

Next, we prove that $ [\sigma]_{A_\infty (\omega)}\lesssim [\sigma]_{\operatorname{Car}(\omega)}$.  Fix $Q_0\in\cd$. Again, let $\cg$ be the stopping family starting at $Q_0$ and defined by
$$\chg(G):=\{G'\in\cd : G'\subseteq G \text{ maximal with } \frac{\sigma(G')}{\omega(G')}>2 \frac{\sigma(G)}{\omega(G)}\}.$$ Then, $1_{\eg(G)}M_{Q_0}^\omega(\sigma)\leq 2 \frac{\sigma(G)}{\omega(G)}$, and $1_{Q_0}=\sum_{\substack{G\in\cg}} 1_{\eg(G)}$ $\omega$-almost everywhere. Moreover, since $\cg$ is $\omega$-sparse, it is $\omega$-Carleson: $$\sum_{\substack{G'\in\cg :\\ G'\subseteq G}}\omega(G)\leq 2 \sum_{\substack{G'\in\cg :\\ G'\subseteq G}}\omega(\eg(G'))=2\omega(\bigcup_{\substack{G'\in\cg :\\ G'\subseteq G}} \eg(G'))\leq 2 \omega(G).$$  Now,
\begin{equation*}
\begin{split}
&\int_{Q_0} M_{Q_0}^\omega(\sigma) \domega= \int_{Q_0} \sum_{\substack{G\in\cg}} 1_{\eg(G)} M_{Q_0}^\omega(\sigma) \domega\leq 2\sum_{\substack{G\in\cg}} \frac{\sigma(G)}{\omega(G)} \omega(\eg(G))\\
&\leq 2 \sum_{\substack{G\in\cg: \\G\subseteq Q_0}} \sigma(G)\leq 2[\sigma]_{\operatorname{Car}(\omega)} \sigma(Q_0).
\end{split}
\end{equation*}
\end{proof}

\section{Weighted characterizations}
In this section, we prove Theorem \ref{thm_twoweight} and Theorem \ref{thm_dualpairingtesting}.
\subsection{Particular family of stopping cubes}\label{sec_particularfamily}

\begin{lemma}[Properties of a particular stopping family]\label{lem_stoppingfamily}Let $E$ be a Banach lattice. Let $\mu$ be a locally finite Borel measure. Let $\cd$ be a finite collection of dyadic cubes. Let $f:\br^d\to E_+$ be a locally integrable, positive function. 

For each dyadic cube $F\in\cd$, the stopping children $\chf(F)$ of $F$ is defined as the collection of all the maximal dyadic cubes $F'\in\{F'\in\cd : F'\subseteq F\}$ that satisfy the stopping condition
\begin{equation}\label{eq_lemma_strongstoppingcondition}
\abs{\sup_{\substack{Q\in\cd:\\Q\supseteq F'}}\angles{f}^\mu_Q}_E>2 \angles{\abs{ \sup_{Q\in\cd}\angles{f}^\mu_Q 1_Q}_E}^\mu_F.
\end{equation}
Let $\cf$ be the stopping family defined by the stopping children $\chf$. For each $F\in \cf$, define the auxiliary function $$f_F:=\sup_{\pif(Q)=F} \angles{f}^\mu_Q 1_Q.$$
Then, the following conditions are satisfied:
\begin{itemize}
\item[a)] The collection $\cf$ is sparse.
\item[b)] Each auxiliary function $f_F$ satisfies the $L^\infty$ estimate
\begin{equation}\label{eq_lemma_auxiliaryfunction}
\norm{f_F}_{L^\infty_E}\leq 2 \angles{\abs{ \sup_{Q\in\cd}\angles{f}^\mu_Q 1_Q}_E}^\mu_F.
\end{equation}
\item[c)] Each auxiliary function $f_F$ satisfies the replacement rule
$$
\int_Q f \dmu \leq \int_Q f_F \dmu\quad\text{ whenever $\pif(Q)=F$}.
$$
\end{itemize}
\end{lemma}
\begin{proof}

First, we check that each auxiliary function satisfies the $L^\infty$ estimate. We note that the condition $\pif(Q)=F$ implies that $Q$ satisfies the opposite of the stopping condition. Now, fix $x\in \bigcup_{Q\in\cd: \pif(Q)=F} Q$. Let $Q_x$ be the minimal (which exists since the collection $\cd$ is finite) dyadic cube such that $\pif(Q_x)=F$ and $Q\ni x$. Since the cube $Q_x$ satisfies the opposite of the stopping condition  \eqref{eq_lemma_strongstoppingcondition}, we have 
$$
\abs{f_F(x)}_E=\abs{\sup_{\substack{Q\in\cd:  \\\pif(Q)=F,\\ Q\ni x}} \angles{f}^\mu_Q}_E\leq \abs{\sup_{\substack{Q\in\cd:  \\Q\supseteq Q_x}} \angles{f}^\mu_Q}_E\leq  2 \angles{\abs{ \sup_{Q\in\cd}\angles{f}^\mu_Q 1_Q}_E}^\mu_F.
$$

Next, we check that $\cf$ is sparse. By the stopping condition \eqref{eq_lemma_strongstoppingcondition},
\begin{equation*}
\begin{split}
\angles{\abs{ \sup_{Q\in\cd}\angles{f}^\mu_Q 1_Q}_E}^\mu_F&\geq \sum_{F'\in\chf(F)} \frac{\mu(F')}{\mu(F)}\angles{\abs{ \sup_{Q\in\cd}\angles{f}^\mu_Q 1_Q}_E}^\mu_{F'}\geq \sum_{F'\in\chf(F)} \frac{\mu(F')}{\mu(F)}\abs{ \sup_{\substack{Q\in\cd:\\Q\supseteq F'}}\angles{f}^\mu_Q}_E\\
&\geq 2 \angles{\abs{ \sup_{Q\in\cd}\angles{f}^\mu_Q 1_Q}_E}^\mu_F \sum_{F'\in\chf(F)} \frac{\mu(F')}{\mu(F)}.
\end{split}
\end{equation*}
Dividing out the factor $\angles{\abs{ \sup_{Q\in\cd}\angles{f}^\mu_Q 1_Q}_E}^\mu_F$ yields $\sum_{F'\in\chf(F)} \mu(F')\leq \frac{1}{2}\mu(F)$.

Finally, we observe that the replacement follows from positivity: $$\int_Q f \dmu = \int_Q \angles{f}_Q^\mu 1_Q \dmu\leq \int_Q f_F \dmu.$$

\end{proof}
\begin{remark}
Instead of the stopping condition \eqref{eq_lemma_strongstoppingcondition}, we could use the stopping condition
\begin{equation}\label{eq_weakstoppingcondition}
\abs{\sup_{\substack{Q\in\cd:\\F\supseteq Q\supseteq F'}}\angles{f}^\mu_Q}_E \geq 2 \norm{\bar{M}^\mu}_{L^1_E(\mu)\to L^{1,\infty}_E(\mu)} \angles{\abs{f}_E}^\mu_F,
\end{equation}
which in the real-valued case (that is, $E=\br$) coalesces with the Muckenhoupt--Wheeden principal cubes stopping condition $\abs{\angles{f}}^\mu_{F'}>2 \angles{\abs{f}}^\mu_F$. 
The stopping family defined by the condition \eqref{eq_weakstoppingcondition} is sparse, because
$$
\sum_{F'\in\chf(F)} \mu(F')\leq \mu\big(\{ \abs{\bar{M}^\mu (1_Ff)}_E> 2 \norm{\bar{M}^\mu}_{L^1_E(\mu)\to L^{1,\infty}_E(\mu)} \angles{\abs{f}_E}^\mu_F\}\big)\leq \frac{1}{2}\mu(F),
$$
and the auxiliary function $f_F:=\sup_{\pif(Q)=F} \angles{f}^\mu_Q 1_Q$ associated with the stopping family satisfies the estimate
$$
\norm{f_F}_{L^\infty_E}\leq 2 \norm{\bar{M}^\mu}_{L^1_E(\mu)\to L^{1,\infty}_E(\mu)} \angles{\abs{f}_E}^\mu_F,
$$
because of a similar argument as in the proof of Lemma \ref{lem_stoppingfamily}.
\end{remark}
\subsection{Proof of the two weight characterization}
In this subsection, we prove Theorem \ref{thm_twoweight}. 
\begin{proof}We prove the norm estimate \eqref{eq_norminequality} by using duality. Let $f\in\lp$ be such that $f\geq 0$, and $g\in\lqp$ be such that $g\geq 0$. By writing out the definition of the operator,
\begin{equation*}
\begin{split}
S:=\int g T(f\sigma) \domega=\sum_{Q\in\cd} \int_Q g \domega \lambda_Q \int_Q f \dsigma.
\end{split}
\end{equation*}

First, we define stopping families.
Associated with $f\in\lp$, let $\cf$ be the stopping family defined by the stopping children
$$
\chf(F):=\{F'\in\cd : F'\subsetneq F \text{ maximal with } \abs{\sup_{\substack{Q\in\cd :\\ Q\supseteq F'}}\angles{f}^\sigma_{Q}}_C>2 \angles{\abs{ \sup_{Q\in\cd} \angles{ f }^\sigma_Q 1_Q }_C}_F^\sigma \}.
$$
 Similarly, let $\cg$ be the stopping family associated with $g\in\lqp$.

 Next, we rearrange the summation by means of the stopping cubes. We use the notation $\pi(Q)=(F,G)$ to indicate that $\pif(Q)=F$ and $\pig(Q)=G$. We have
\begin{equation}
\begin{split}\label{eq_rearrangement}
S:=\sum_{Q\in\cd}=\sum_{F\in\cf, G\in\cg} \sum_{\substack{Q\in\cd:\\\pi(Q)=(F,G)}}&\overset{i)}{=} \Big(\sum_{F\in\cf} \sum_{\substack{G\in\cg:\\G\subseteq F}} +\sum_{G\in\cg} \sum_{\substack{F\in\cf:\\F\subsetneq G}}\Big) \sum_{\substack{Q\in\cd:\\\pi(Q)=(F,G)}}\\
&\overset{ii)}{\leq}\Big(\sum_{F\in\cf} \sum_{\substack{G\in\cg:\\ \pif(G)=F}}+\sum_{G\in\cg} \sum_{\substack{F\in\cf:\\\pig(F)=G}}\Big) \sum_{\substack{Q\in\cd:\\\pi(Q)=(F,G)}}\\
&=:S_{G\subseteq F}+S_{G\supseteq F},
\end{split}
\end{equation}
because of the following observations:
\begin{itemize}
\item[i)] Under the condition $\pi(Q)=(F,G)$, we have $F\cap G\neq \emptyset$. Hence, by dyadic nestedness, either $G\subseteq F$ or $G\supsetneq F$.
\item[ii)] Under the conditions $\pi(Q)=(F,G)$ and $G\subseteq F$, we have $Q\subseteq G\subseteq F$. Hence $F=\pif(Q)\subseteq \pif(G)\subseteq \pif(F)=F$, which implies that $\pif(G)=F$. Similarly, when $\pi(Q)=(F,G)$ and $F\subsetneq G$, we have $\pig(F)=G$.
\end{itemize}

By symmetry, it suffices to consider the  summation $S_{G\subseteq F}$ in the inequality \eqref{eq_rearrangement}. 
Under the condition $\pig(Q)=(F,G)$, we can write
\begin{equation}\label{eq_replacements}
\int_Q g \domega \lambda_Q \int_Q f \dsigma\leq \int_Q g_G \domega \lambda_Q \int_Q f_F \dsigma=\int g_G \lambda_Q \big( \int_Q f_F \dsigma \big) 1_Q \domega,
\end{equation}
where
\begin{subequations}
\begin{align}
\nonumber g_G&:=\sum_{G'\in\chg(G)} \angles{g}^\omega_{G'} 1_{G'} +g 1_{\eg(G)}, \\\nonumber
f_F&:=\sup_{\substack{Q\in\cd:\\\pif(Q)=F}}\angles{f}^\sigma_Q 1_{Q},
\end{align}
\end{subequations}
which follows from the following observations:
\begin{itemize}
\item If $G'\in\chg(G)$ is such that $G'\cap Q\neq \emptyset$, then, by dyadic nestedness, either $G'\subsetneq Q$ or $Q\subseteq G'$, the latter of which is excluded by the condition $\pig(Q)=G$. Therefore
$$
\int_Q g 1_{G'} \domega=\int_Q \angles{g}^\omega_{G'} 1_{G'} \domega,
$$
which implies that
\begin{equation}\label{eq_replacement}
\begin{split}
\int_Q g \domega &=\int_Q (\sum_{G'\in\chg(G)} 1_{G'} + 1_{\eg(G)}) g\domega\\
&=\int_Q  \big(\sum_{G'\in\chg(G)} \angles{g}^\omega_{G'} 1_{G'} +1_{\eg(G)}  g\big) \domega=:\int_Q g_G \domega.
\end{split}
\end{equation}
\item By positivity,
$$
\int_Q f \dsigma= \int_Q \angles{f}^\sigma_Q 1_{Q} \dsigma \leq \int_Q \big( \sup_{\substack{Q\in\cd:\\\pif(Q)=F}}\angles{f}^\sigma_Q 1_{Q}\big) \dsigma.
$$
\end{itemize}

Combining \eqref{eq_rearrangement} and \eqref{eq_replacements} yields, by positivity,
$$
S_{G\subseteq F}\leq \sum_{F\in\cf} \int \big(\sum_{\substack{G\in\cg:\\ \pif(G)=F}} g_G \big) \big(\sum_{\substack{Q\in\cd: \\Q\subseteq F}} \lambda_Q \int_Q f_F \dsigma 1_Q \big)\domega.
$$
By definition, $T_F(f_F\sigma):= \sum_{\substack{Q\in\cd: \\Q\subseteq F}} \lambda_Q \int_Q f_F \dsigma 1_Q$. We write $G_F:= \sum_{\substack{G\in\cg:\\ \pif(G)=F}} g_G$.
By H\"older's inequality, the direct $L^\infty$ testing condition \eqref{eq_testing_direct}, and H\"older's inequality with the exponents $p$ and $q'$ (which holds because, by assumption, $\frac{1}{p}+\frac{1}{q'}\geq 1$), we obtain
\begin{equation}\label{eq_intermediate}
\begin{split}\
S_{G\subseteq F}&\leq \sum_{F\in\cf} \int G_F T_F(f_F\sigma)\domega\\
&\leq  \sum_{F\in\cf} \norm{G_F}_{\lqp} \norm{T_F(f_F\sigma)}_{\llq}\\
&\leq  \mathfrak{T}\sum_{F\in\cf} \norm{G_F}_{\lqp} \norm{f_F}_{L^\infty_C} \sigma(F)^{1/p}\\
&\leq \mathfrak{T}\Big(\sum_{F\in\cf} \norm{G_F}_{\lqp}^{q'} \Big)^{1/{q'}} \Big(\sum_{F\in\cf}\norm{f_F}_{L^\infty_C}^p \sigma(F) \Big)^{1/p}.
\end{split}
\end{equation}

Next, we estimate the second factor in the right-most side of the inequality \eqref{eq_intermediate}.
We now invoke the properties of the stopping cubes that are stated in Lemma \ref{lem_stoppingfamily}: The auxiliary function $f_F$ satisfies the $L^\infty$ estimate $$\norm{f_F }_{L^\infty_{C}}\leq 2 \angles{\abs{\bar{M}^\sigma f}_{C}}^\sigma_F,$$
and the collection $\cf$ is $\sigma$-sparse. Therefore, by the dyadic Carleson embedding theorem (Lemma \ref{lem_carlesonembedding}), and by the universal bound for the dyadic lattice Hardy--Littlewood maximal function (Theorem \ref{thm_universalbound}), we obtain
\begin{equation}\label{followingverbatim}
\begin{split}
&\Big( \sum_{F\in\cf} \norm{f_F}_{L^\infty_C}^p \sigma(F)\Big)^{1/p}\leq 2 \Big( \sum_{F\in\cf} \angles{\abs{\bar{M}^\sigma f}_C}^p \sigma(F)\Big)^{1/p}\leq 4p' \norm{\bar{M}^\sigma f}_{L^p_C(\sigma)}\\
&\leq 4p' \norm{\bar{M}^\sigma}_{L^p_C(\sigma)\to L^p_C(\sigma)} \norm{f}_{L^p_C(\sigma)}\lesssim_p \norm{\bar{M}}_{L^p_C\to L^p_C} \norm{f}_{L^p_C(\sigma)}.
\end{split}
\end{equation}

Finally, we estimate the first factor in the right-most side of the inequality \eqref{eq_intermediate}. Again, the collection $\cg$ is $\omega$-sparse. Using the $L^p$-variant of Pythagoras' theorem (Lemma \ref{lem_pythagoras}), and the rearrangement $\sum_{F\in\cf} \sum_{\substack{G\in\cg:\\ \pif(G)=F}}=\sum_{G\in\cg}$ yields
\begin{equation*}
\begin{split}
\Big(\sum_{F\in\cf} \norm{\sum_{\substack{G\in\cg:\\ \pif(G)=F}} g_G}_{\lqp}^{q'} \Big)^{1/{q'}}&\leq 3q' \Big(\sum_{G\in\cg} \norm{g_G}_{\lqp}^{q'} \Big)^{1/{q'}}. 
\end{split}
\end{equation*}
The proof is completed by the estimate
$$
\Big(\sum_{G\in\cg} \norm{g_G}_{\lqp}^{q'} \Big)^{1/{q'}} \leq 3q \norm{g}_{\lqp},
$$
which is checked as Lemma \ref{lem_decomposition}.

\end{proof}

\begin{lemma}\label{lem_decomposition}Let $1< p \leq \infty$. Let $\mu$ be a locally finite Borel measure. Let $E$ be a Banach space. Assume that $\cs$ is a sparse collection of dyadic cubes. Let $$f_S:=\sum_{S\in\chs(S)}\angles{f}_{S'}^\mu 1_{S'}+f1_{\es(S)}.$$ Then
$$
\Big(\sum_{S\in\cs}\norm{f_S}_{L^p_E(\mu)}^p\Big)^{1/p}\leq 3p' \norm{f}_{L^p_E(\mu)}.
$$
\end{lemma}
\begin{proof}Note that, for each $S$, the sets $\{S'\}_{S'\in\chs(S)}$ are pairwise disjoint, and the sets $\{\es(S)\}_{S\in\cs}$ are pairwise disjoint. Therefore, by H\"older's inquality, 
\begin{equation*}
\begin{split}
&\Big(\sum_{S\in\cs }\norm{f_S}_{L^p_E(\mu)}^p\Big)^{1/p}\\
&\leq \Big(\sum_{S\in\cs} \sum_{S'\in\chs(S)} \norm{\angles{f}_{S'}^\mu 1_{S'}}_{L^p_E(\mu)}^p\Big)^{1/p}+\norm{\sum_{S\in\cs} 1_{\es(S)} f}_{L^p_E(\mu)}\\
&\leq \Big(\sum_{S'\in\cs}  \big(\angles{\abs{f}_E}^\mu_{S'}\big)^p \mu(S')\Big)^{1/p}+\norm{f}_{L^p_E(\mu)}.
\end{split}
\end{equation*}
Using the dyadic Carleson embedding theorem (Lemma \ref{lem_carlesonembedding}) completes the proof.

\end{proof}
\subsection{Proof of the $A_\infty$ weights characterization}
In this subsection, we prove Theorem \ref{thm_dualpairingtesting}.
\begin{proof} Following verbatim the beginning of the proof of Theorem \ref{thm_twoweight} (in particular, the stopping families are defined similarly), we arrive at:
\begin{equation}\label{eq_B_rearragnement}
\begin{split}
S&:=\int g T(f\sigma) \domega=\sum_{Q\in\cd} \int_Q g \domega \lambda_Q \int_Q f \dsigma\\
&\leq \Big(\sum_{F\in\cf} \sum_{\substack{G\in\cg:\\ \pif(G)=F}}+\sum_{G\in\cg} \sum_{\substack{F\in\cf:\\\pig(F)=G}}\Big) \sum_{\substack{Q\in\cd:\\\pi(Q)=(F,G)}}\int_Q g \domega \lambda_Q \int_Q f \dsigma\\
&=:S_{G\subseteq F}+S_{G\supseteq F}.
\end{split}
\end{equation}
By symmetry, it suffices to consider the first summation $S_{G\subseteq F}$. 
Under the condition $\pi(Q)=(F,G)$, we obtain, by positivity, that 
\begin{equation}
\begin{split}\label{secondly}
\int_Q g \domega \lambda_Q \int_Q f \dsigma\leq \int_Q g_G \domega \lambda_Q \int_Q f_F \dsigma=\int g_G  \lambda_Q \big( \int_Q f_F  \dsigma \big) 1_Q \domega ,
\end{split}
\end{equation}
where $g_G:=\sup_{\substack{Q\in\cd:\\\pig(Q)=G}} \angles{g}^\omega_{Q} 1_{Q}$, and $f_F:=\sup_{\substack{Q\in\cd:\\\pif(Q)=F}}\angles{f}^\sigma_Q 1_{Q}.
$ 
Combining \eqref{eq_B_rearragnement} and \eqref{secondly} yields, by positivity,
\begin{equation}\label{summations}
\begin{split}
S_{G\subseteq F}&\leq \sum_{F\in\cf} \sum_{\substack{G\in\cg:\\ \pif(G)=F}} \int  g_G \, \big(\sum_{\substack{Q\in\cd:\\Q\subseteq G}} \lambda_Q \int_Q f_F \dsigma 1_Q \big) \domega.
\end{split}
\end{equation}

 By definition, $\sum_{\substack{Q\in\cd:\\Q\subseteq G}} \lambda_Q \int_Q f_F \dsigma 1_Q =:T_G(f_F\,\sigma)$. By the dual pairing $L^\infty$ testing condition \eqref{eq_dualpairingtestingcondition}, and by H\"older's inequality with the exponents $p$ and $q'$ (which holds because, by assumption, $\frac{1}{p}+\frac{1}{q'}\geq 1$) applied twice, we obtain
\begin{equation*}
\begin{split}
&S_{G\subseteq F}:=\sum_{F\in\cf} \sum_{\substack{G\in\cg:\\ \pif(G)=F}} \int  g_G \, T_G(f_F) \domega\\
&\leq \mathfrak{B} \sum_{F\in\cf} \norm{f_F}_{L^\infty_C} \sum_{\substack{G\in\cg:\\ \pif(G)=F}} \sigma(G)^{1/p} \norm{g_G}_{L^\infty_{D^*}} \omega(G)^{1/{q'}}\\
&\leq \mathfrak{B} \sum_{F\in\cf} \norm{f_F}_{L^\infty_C} \Big(\sum_{\substack{G\in\cg:\\ \pif(G)=F}} \sigma(G)\Big)^{1/p} \Big(\sum_{\substack{G\in\cg:\\ \pif(G)=F}} \norm{g_G}_{L^\infty_{D^*}}^{q'} \omega(G)\Big)^{1/{q'}}\\
&\leq \mathfrak{B} \Big( \sum_{F\in\cf} \norm{f_F}_{L^\infty_C}^p \big(\sum_{\substack{G\in\cg:\\ \pif(G)=F}} \sigma(G)\big)\Big)^{1/p} \Big(\sum_{F\in\cf}\sum_{\substack{G\in\cg:\\ \pif(G)=F}} \norm{g_G}_{L^\infty_{D^*}}^{q'} \omega(G)\Big)^{1/{q'}}.
\end{split}
\end{equation*}

Since $\cg$ is $\omega$-sparse, it is $\omega$-Carleson, which follows from the observation
$$
\sum_{\substack{G'\in \cg: \\ G'\subseteq G}}\omega(G')\leq  2\sum_{\substack{G'\in \cg: \\ G'\subseteq G}}\omega(\eg(G'))= 2\omega(\bigcup_{\substack{G'\in \cg: \\ G'\subseteq G}} \eg(G'))\leq 2 \omega(G).
$$ By assumption, $\sigma$ satisfies the $A_\infty$ condition with respect to $\omega$. By Lemma \ref{lem_ainfinitycarleson}, the $\omega$-Carleson collection $\cg$ is also $\sigma$-Carleson. Hence,
$$
\sum_{\substack{G\in\cg:\\ \pif(G)=F}} \sigma(G) \lesssim [\sigma]_{A_\infty(\omega)} \sigma(F).
$$ 
Moreover, $\sum_{F\in\cf}\sum_{\substack{G\in\cg:\\ \pif(G)=F}}=\sum_{G\in\cg}$. Altogether,
$$
S_{G\subseteq F} \leq 8  \mathfrak{B}  [\sigma]_{A_\infty(\omega)}^{1/p}\,\Big( \sum_{F\in\cf} \norm{f_F}_{L^\infty_C}^p \sigma(F)\Big)^{1/p} \Big(\sum_{G\in\cg} \norm{g_G}_{L^\infty_{D^*}}^{q'} \omega(G)\Big)^{1/{q'}}.
$$
The proof is completed by estimating each factor on the right-hand side of this inequality as in \eqref{followingverbatim}. 
\end{proof}

\section{Unweighted characterization under alternative assumptions}

In this section, we prove Theorem \ref{thm_alternativetesting}. First, we reduce the theorem  to the existence of an auxiliary collection $\cf$ of dyadic cubes, and an auxiliary family $\{f_F\}_{F\in\cf}$  of functions (Lemma \ref{lemma_existenceofcollection}).  Then, we construct these auxiliary quantities by using stopping conditions.
\subsection{Reduction to the existence of a stopping family}
\begin{lemma}[Reduction of the characterization]\label{lemma_existenceofcollection}Let $E$ be a Banach lattice. Let $1<p< t\leq \infty$. Let $f:\br^d\to E_+$ be a non-negative, locally integrable function. 

Assume that there exists a collection $\cf$ of dyadic cubes and a family $\{f_F\}_{F\in\cf}$ of auxiliary functions that satisfy the following properties:
\begin{itemize}
\item[a)] The family $\{f_F\}_{F\in\cf}$ satisfies the replacement rule:
\begin{equation}\label{eq_temporaryassumption_replacement}
\int_Q f \dmu \leq \int_Q f_F \dmu\quad \text{whenever $Q\in\cd$ and $F\in\cf$ such that $ \pif(Q)=F$}.
\end{equation}

\item[b)] The family $\{f_F\}_{F\in\cf}$ satisfies the norm estimate:
\begin{equation}\label{eq_temporaryassumption_averageestimate}
\norm{f_F}_{L^t_E(\mu)}\lesssim \angles{\abs{I f}_E}_F^\mu \mu(F)^{1/t}\quad\text{for every $F\in\cf$}.
\end{equation}
Here, $I:L^p_E(\mu) \to L^p_E(\mu)$ is an auxiliary operator that is bounded with $\norm{I}_{L^p_E(\mu) \to L^p_E(\mu)}\lesssim 1$. For example, $I$ can be the identity operator. 
\item[c)] We have the norm estimate:
\begin{equation}\label{condition_additional}
\norm{\sum_{\substack{Q\in\cd:\\\pif(Q)=F}}\lambda_Q \int_Q f \dmu 1_Q}_{L^\infty_E}\leq 4 \angles{\abs{T(f\mu)}_E}_F.
\end{equation}
\item[d)] The collection $\cf$ is sparse.
\end{itemize} 
Furthermore, assume that the operator $T(\cdotroomy \mu):L^p_E(\mu)\to L^p_E(\mu)$ satisfies the endpoint $L^t$ testing condition:
\begin{equation}\label{endpointlttestingcondition}
\norm{T_R(f\mu)}_{L^1_E(\mu)}\leq \mathfrak{B}_t  \norm{f}_{L^t_E(R,\mu)}\mu(R)^{1-1/t}
\end{equation}
for every $R\in\cd$, and $f\in L^t_E(R,\mu)$.

 Then, we have the norm estimate
$$
\norm{T(f\mu)}_{L^p_E(\mu)}\lesssim_p \mathfrak{B}_t\norm{f}_{L^p_E}.
$$
\end{lemma}
\begin{proof}[Proof of Lemma \ref{lemma_existenceofcollection}] 

By the $L^p$ variant of Pythagoras' theorem (Lemma \ref{lem_pythagoras}), and by the replacement rule \eqref{eq_temporaryassumption_replacement}, we obtain
\begin{equation*}
\begin{split}
\norm{T(f\mu)}_{L^p_E(\mu)}&=\norm{\sum_{Q\in\cd} \lambda_Q \int_Q f\dmu 1_Q}_{L^p_E(\mu)}\\
&\lesssim_p \Big(\sum_{F\in\cf}\norm{\sum_{\substack{Q\in\cd:\\\pif(Q)=F}} \lambda_Q \int_Q f \dmu 1_Q }_{L^p_E(\mu)}^p\Big)^{1/p}\\
&\leq \Big(\sum_{F\in\cf}\norm{\sum_{\substack{Q\in\cd:\\\pif(Q)=F}} \lambda_Q \int_Q f \dmu 1_Q }_{L^\infty_E(\mu)}^{p-1} \norm{\sum_{\substack{Q\in\cd:\\\pif(Q)=F}} \lambda_Q \int_Q f_F \dmu 1_Q }_{L^1_E(\mu)}\Big)^{1/p}.
\end{split}
\end{equation*}
The first factor is estimated by the norm estimate \eqref{condition_additional}. For the second factor, from the endpoint $L^t$ testing condition \eqref{endpointlttestingcondition}, and the norm estimate for the auxiliary functions \eqref{eq_temporaryassumption_averageestimate}, it follows that
\begin{equation*}
\begin{split}
&\norm{\sum_{\substack{Q\in\cd:\\\pif(Q)=F}} \lambda_Q \int_Q f_F \dmu 1_Q }_{L^1_E(\mu)}\leq \norm{T_F(f_F \mu)}_{L^1_E(\mu)}\\
&\leq \mathfrak{B}_t \norm{f_F}_{L^t_E(\mu)} \mu(F)^{1-1/t}\lesssim \mathfrak{B}_t \sum_{F\in\cf} \angles{\abs{If}_E}^\mu_F  \mu(F).
\end{split}
\end{equation*}
Altogether,
\begin{equation*}
\begin{split}
\norm{T(f\mu)}_{L^p_E(\mu)}&\lesssim_p \mathfrak{B}_t^{1/p} \Big(\sum_{F\in\cf} \angles{\abs{T(f\mu)}_E}^{p-1}_F \mu(F)^{1/p'} \angles{\abs{If}}^\mu_F  \mu(F)^{1/p} \Big)^{1/p}.
\end{split}
\end{equation*}
By H\"older's inequality, the dyadic Carleson embedding theorem (Lemma \ref{lem_carlesonembedding}), and the assumption that $\norm{I}_{L^p_E(\mu)\to L^p_E(\mu)}\lesssim 1$, we obtain
\begin{equation*}
\begin{split}
\norm{T(f\mu)}_{L^p_E(\mu)}&\leq  \mathfrak{B}_t^{1/p} \bigg(\Big( \sum_{F\in\cf} \angles{\abs{T(f\mu)}_E}^{p}_F \mu(F)\Big)^{1/p} \bigg)^{1/p'}  \bigg(\Big( \sum_{F\in\cf} \big(\angles{\abs{If}}^\mu_F\big)^p  \mu(F) \Big)^{1/p} \bigg)^{1/p}\\
&\lesssim \norm{T(f\mu)}_{L^p_E(\mu)}^{1/p'} \big(\mathfrak{B}_t \norm{If}_{L^p_E(\mu)}\big)^{1/p}\\
&\lesssim \norm{T(f\mu)}_{L^p_E(\mu)}^{1/p'} \big(\mathfrak{B}_t \norm{f}_{L^p_E(\mu)}\big)^{1/p}. 
\end{split}
\end{equation*}
Dividing out the factor $\norm{T(f\mu)}_{L^p_E(\mu)}^{1/p'}$ completes the proof.
\end{proof}
\subsection{Table of stopping families}

Note that we can use multiple stopping conditions in order to use multiple auxiliary families of  functions, while keeping the estimate for each family of auxiliary functions and keeping the measure  condition (sparseness).  This is based on the following observations. Let $A$ and $B$ be conditions for cubes. (By {\it a condition for cubes} it is meant a condition such that of each cube it can be said whether the cube satisfies the condition or not.)
\begin{itemize}
\item (Keeping the measure condition) If $\ch_{\cf_A}(F)$ is the collection of all the maximal $F'\in\{F'\in\cd : F'\subseteq F\}$ that satisfy the  condition A, and $\ch_{\cf_{B}}(F)$ is the collection of all the maximal  $F'\in\{F'\in\cd : F'\subseteq F\}$ that satisfy the condition $B$, then the collection $\chf(F)$ of all the maximal $F'\in\{F'\in\cd : F'\subseteq F\}$ that satisfy the condition $A$ or the condition $B$ is the union $\chf(F)=\ch_{\cf_A}(F)\bigcup \ch_{\cf_B}(F)$. We have the measure condition:
$$
\sum_{F'\in \chf(F)}\mu(F')\leq \sum_{F'\in \ch_{\cf_A}(F)}\mu(F')+ \sum_{F'\in \ch_{\cf_B}(F)}\mu(F').
$$
\item (Keeping the estimate for each family of auxiliary functions) If $Q\in\{Q\in\cd: Q\subseteq F\}$ is such that $Q\subseteq F'$ for no $F'\in \ch_{\cf_A}(F)\bigcup \ch_{\cf_B}(F)$, then, by maximality, $Q$ satisfies neither the condition A nor the condition B. 
\end{itemize}
Now, by the reduction (Lemma \ref{lemma_existenceofcollection}), Theorem \ref{thm_alternativetesting} follows from using the stopping conditions of Table \ref{table_stopping conditionss}, tailored for each assumption: 
\begin{itemize}
\item[i)] Assume the Hardy--Littlewood property: Use the stopping condition and the auxiliary family A together with the stopping condition D.  (That is, the stopping children $\chf(F)$ of $F$ is defined as the collection of all the maximal dyadic cubes $F'\in\{F'\in Q: F'\subseteq F\}$ that satisfy the stopping condition A or the stopping condition D. The auxiliary collection $\cf$ is the collection defined by $\chf$. The auxiliary family $\{f_F\}_{F\in\cf}$ is the family A.) 
\item[ii)] Assume that the measure is doubling: Use the stopping condition and the auxiliary family B together with the stopping condition D.
\item[iii)] Assume the $L^t$ testing condition:  Use the stopping condition and the auxiliary family C together with the stopping condition D.
\end{itemize}

\begin{table}
\small
\caption{Let $E$ be a Banach lattice, $\mu$ a locally finite Borel measure, and $f:\br^d\to E_+$ a positive, locally integrable function. Let $F\in \cd$. The stopping children $\chf(F)$ of $F$ determined by a stopping condition is defined as the collection of all the maximal $F'\in\{F'\in\cd: F'\subseteq F\}$ that satisfy the stopping condition. The family $\{Q\in\cd :\pif(Q)=F\}$ is the collection of all $Q\in\{Q\in\cd: Q\subseteq F\}$ such that $Q\subseteq F'$ for no $F'\in\chf(F)$. In particular, $ \pif(Q)=F$ implies that $Q$ does not satisfy the stopping condition. The properties listed in the table are proven in Lemma \ref{lem_stoppingfamily}, Lemma \ref{stoppingfamily_bootstrap}, and Lemma \ref{lem_stoppingfamilyprincipal}.}
\label{table_stopping conditionss}
\begin{tabular}{ |l l l| }
 \hline
 \multicolumn{3}{|l|}{}\\
\multirow{3}{*}{A} & Stopping condition  & $\abs{\sup_{\substack{Q\in\cd:\\Q\supseteq F'}}\angles{f}^\mu_Q}_E>4 \angles{\abs{\mbar^\mu f}_E}^\mu_F$.\\[0.3cm]
  
  &Auxiliary function &$f_F:= \sup_{\substack{Q\in\cd:\\\pif(Q)=F}} \angles{f}^\mu_Q 1_Q.$ \\[0.3cm]
  & Estimate&  $\norm{f_F}_{L^\infty_E(\mu)}\leq4 \angles{\abs{\bar{M}^\mu f}_E}^\mu_F$. \\[0.3cm]
\hline
 \multicolumn{3}{|l|}{}\\
  \multirow{3}{*}{A'}&Stopping condition & $
\abs{\sup_{\substack{Q\in\cd:\\F\supseteq Q\supseteq F'}}\angles{f}^\mu_Q}_E \geq 4 \norm{\bar{M}^\mu}_{L^1_E(\mu)\to L^{1,\infty}_E(\mu)} \angles{\abs{f}_E}^\mu_F
$\\[0.3cm]
&  Auxiliary function &$f_F:= \sup_{\substack{Q\in\cd:\\\pif(Q)=F}} \angles{f}^\mu_Q 1_Q.$ \\[0.3cm]
 & Estimate&  $\norm{f_F}_{L^\infty_E}\leq 4 \norm{\bar{M}^\mu}_{L^1_E(\mu)\to L^{1,\infty}_E(\mu)} \angles{\abs{f}_E}^\mu_F$ \\[0.3cm]
 \hline
  \multicolumn{3}{|l|}{}\\
\multirow{4}{*}{B} & Stopping condition & $\angles{\abs{f}_E}_{F'}^\mu > 4 \angles{\abs{f}_E}^\mu_{F}$.\\[0.3cm]
  &Auxiliary function &$f_F:=\sum_{F'\in\chf(F)} \angles{f}^\mu_{F'}1_{F'} +f1_{E(F)}$ \\[0.3cm]
 & Estimate&  $\norm{f_F}_{L^\infty_E}\leq 4 \big( \sup_{F'\in\chf(F)} \frac{\mu(\hat{F'})}{\mu(F')}\big) \angles{\abs{ f}_E}^\mu_F$,\\[0.3cm]
   & &where $\hat{F'}$ denotes the dyadic parent of $F'$. \\[0.3cm]
    \hline
     \multicolumn{3}{|l|}{}\\
 \multirow{4}{*}{C} & Stopping condition & $ \angles{\abs{f}_E}_{F'}^\mu > 4 \angles{\abs{f}_E}^\mu_{F}.$\\[0.3cm]
  &Auxiliary function &$f_F:=\sum_{F'\in\chf(F)} \int_{F'} f \dmu \frac{1_{\hat{F'}}}{\mu(\hat{F'})}+f1_{E(F)}$, \\[0.3cm]
 & &where $\hat{F'}$ denotes the dyadic parent of $F'$.\\[0.3cm]
 & Estimate&  $\norm{f_F}_{L^t}\lesssim_t \angles{\abs{ f}_E}^\mu_F \mu(F)^{1/t}.$\\[0.3cm]
  \hline \multicolumn{3}{|l|}{}\\
     \multirow{3}{*}{D} & Stopping condition & $\abs{\sum_{\substack{Q\in\cd: \\Q\supseteq F'}} \lambda_Q \int_Q f \dmu}_E>4 \angles{\abs{T_\lambda (f \mu)}_E}_F.$ \\[0.4cm]
     & Auxiliary function & $f_F:=\sum_{\substack{Q\in\cd:\\\pif(Q)=F}}\lambda_Q \int_Q f \dmu 1_Q$.\\[0.4cm]
  & Estimate&  $\norm{f_F}_{L^\infty_E(\mu)}\leq 4 \angles{\abs{T_\lambda (f \mu)}_E}_F.$ \\[0.4cm]
    \hline  \multicolumn{3}{|l|}{}\\
    \multicolumn{3}{| p{\textwidth} |}{In the cases A, B, and C, the auxiliary function $f_F$ satisfies the replacement rule: $$\int_Q f\dmu \leq \int_Q f_F \dmu\text{ whenever } \pif(Q)=F.$$
    The stopping children $\chf(F)$ determined by each stopping condition satiesfies the measure condition (sparseness): 
    $$
    \sum_{F'\in\chf(F)}\mu(F')\leq \frac{1}{4}\mu(F).
    $$}\\[0.5cm]
    \hline
\end{tabular}
\end{table}
\newpage
\begin{lemma}[Particular stopping family]\label{stoppingfamily_bootstrap}Let $\mu$ be a locally finite Borel measure. Let $E$ be a Banach lattice. Let $\cd$ be a finite collection of dyadic cubes. Let $f:\br^d\to E$ be a non-negative, locally integrable function. 

For each $F\in\cd$, the stopping children $\ch(F)$ of $F$ is defined as the collection of all the maximal $F'\in\{F'\in\cd : F'\subseteq F\}$ that satisfy the stopping condition
\begin{equation}\label{auxiliary_stoppingcondition}
\abs{\sum_{\substack{Q\in\cd: \\Q\supseteq F'}} \lambda_Q \int_Q f \dmu}_E>4 \angles{\abs{\sum_{Q\in\cd}\lambda_Q \int_Q f \dmu 1_Q}_E}_F.
\end{equation}
Recall that $\{Q\in\cd: \pif(Q)=F\}$ denotes the collection of all $Q\in \{Q\in\cd : Q\subseteq F \}$ such that $Q\subseteq F'$ for no $F'\in\chf(F)$. 

Then, 
\begin{equation}\label{stopfam_measure}
\sum_{F'\in\chf(F)} \mu(F')\leq \frac{1}{4}\mu(F),
\end{equation}
and
\begin{equation}\label{stopfam_estimate}
\norm{\sum_{\substack{Q\in\cd:\\\pif(Q)=F}}\lambda_Q \int_Q f \dmu 1_Q}_{L^\infty_E}\leq 4 \angles{\abs{\sum_{Q\in\cd}\lambda_Q \int_Q f \dmu 1_Q}_E}_F.
\end{equation}
\end{lemma}
\begin{proof}First, we check \eqref{stopfam_measure}. By the stopping condition \eqref{auxiliary_stoppingcondition}, 
\begin{equation*}
\begin{split}
\angles{\abs{\sum_{Q\in\cd}\lambda_Q \int_Q f \dmu 1_Q}_E}_F&\geq \sum_{F'\in\chf(F)}  \frac{\mu(F')}{\mu(F)} \abs{\sum_{\substack{Q\in\cd:\\Q\supseteq F'}}\lambda_Q \int_Q f \dmu }_E\\
&\geq  \sum_{F'\in\chf(F)} \frac{\mu(F')}{\mu(F)}4 \angles{\abs{\sum_{Q\in\cd}\lambda_Q \int_Q f \dmu 1_Q}_E}_F.
\end{split}
\end{equation*}
Dividing out the factor $\angles{\abs{\sum_{Q\in\cd}\lambda_Q \int_Q f \dmu 1_Q}_E}_F$ yields $\sum_{F'} \mu(F')\leq \frac{1}{4} \mu(F)$.

Finally, we check \eqref{stopfam_estimate}. Fix  $x\in\bigcup_{Q\in\cd: \pif(Q)=F}$. Let $Q_x\in\cd$ be the minimal dyadic cube (which exists because, by assumption, the collection $\cd$ is finite) such that $\pif(Q)=F$ and $Q\ni x$. Note that $\pif(Q)=F$ implies that $Q$  does not satisfy thestopping condition \eqref{auxiliary_stoppingcondition}. Therefore,
$$
\abs{\sum_{\substack{Q\in\cd:\\\pif(Q)=F}} \lambda_Q \int_Q f \dmu 1_Q(x)}_E\leq  \abs{\sum_{\substack{Q\in\cd:\\ Q\supseteq Q_x}} \lambda_Q \int_Q f \dmu}\leq 4 \angles{\abs{\sum_{Q\in\cd}\lambda_Q \int_Q f \dmu 1_Q}_E}_F.
$$
\end{proof}

A collection $\cd$ of dyadic cubes is {\it a truncated dyadic system} if $$\cd=\{Q : Q\subseteq Q_0, \ell(Q)\geq 2^{-N} \ell(Q_0)\}$$
for some dyadic cube $Q_0$ and some non-negative integer $N$. Let $\cd_*$ denote the collection of all the minimal dyadic cubes in a collection $\cd$ of dyadic cubes. Define the finest averaging by $$\be_{\cd_*}^\mu f:=\sum_{Q\in\cd_*} \angles{f}_Q^\mu 1_Q.$$

\begin{lemma}[Properties of {\it the Muckenhoupt--Wheeden principal cubes}]\label{lem_stoppingfamilyprincipal}Let $E$ be a Banach lattice. Let $\mu$ be a locally finite Borel measure. Let $\cd$ be a truncated dyadic system. Let $f:\br^d\to E_+$ be a locally integrable, non-negative function. 

For each dyadic cube $F\in\cd$, the stopping children $\chf(F)$ of $F$ is defined as the collection of all the maximal dyadic cubes $F'\in\{F'\in\cd : F'\subseteq F \}$ that satisfy the stopping condition
\begin{equation}
\angles{\abs{f}_E}_{F'}^\mu>2 \angles{\abs{f}_E}^\mu_F.
\end{equation}
Recall that $\{Q\in\cd: \pif(Q)=F\}$ denotes the collection of all $Q\in \{Q\in\cd : Q\subseteq F \}$ such that $Q\subseteq F'$ for no $F'\in\chf(F)$. Then, \begin{itemize}
\item[a)]The stopping children are sparse:
$$\sum_{F'\in\chf(F)} \mu(F')\leq \frac{1}{2}\mu(F).$$
\item[b)]The terms of the auxiliary functions satisfy the norm estimates: 
\begin{subequations}
\begin{align}
\label{propertyfirst}\norm{\be_{\cd_*}^\mu f1_{\ef(F)}}_{L^\infty_E}&\lesssim  \angles{\abs{f}_E}_F^\mu\\
\label{propertysecond}\norm{\sum_{F'\in\chf(F)} \angles{f}_{F'}^\mu1_{F'}}_{L^\infty_E}&\lesssim \big(\sup_{F'\in\chf(F)}\frac{\mu(\hat{F'})}{\mu(F')} \big)\angles{\abs{f}_E}_F^\mu\\
\label{propertythird}\norm{\sum_{F'\in\chf(F)} \int_{F'} f \dmu \frac{1_{\hat{F'}}}{\mu(\hat{F'})}}_{L^t_E}&\lesssim_t \angles{\abs{f}_E}_F^\mu \mu(F)^{1/t},
\end{align}
\end{subequations}
where $\hat{F'}$ denotes the dyadic parent of $F'$.
\item[c)]The auxiliary functions satisfy the replacement rules: 
\begin{equation}
\label{eq_principal_replacament}
\int_Q f \dmu \leq \int_Q f_F \dmu\quad\text{ whenever $\pif(Q)=F$},
\end{equation}
for the auxiliary function $$f_F:=\sum_{F'\in\chf(F)} \angles{f}^\mu_{F'} 1_{F'}+f1_{\ef(F)},$$ and for the auxiliary function $$f_F:=\sum_{F'\in\chf(F)} \int_{F'} f \dmu \frac{1_{\hat{F'}}}{\mu(\hat{F'})}+f1_{\ef(F)}.$$
\end{itemize}
\end{lemma}
\begin{proof}
First, we check the inequality \eqref{propertyfirst}.  By maximality, if $ Q\subseteq F$ satisfies $\angles{\abs{f}_E}_{Q}^\mu > 2 \angles{\abs{f}_E}^\mu_{F}$, then $Q\subseteq F'$ for some $F'\in\chf(F)$. By contraposition, if $Q\subseteq F$ and there is no $F'\in\ch(F)$ such that $Q\subseteq F'$, then $Q$ satisfies $\angles{\abs{f}_E}_{Q}^\mu \leq 2 \angles{\abs{f}_E}^\mu_{F}$. Note that $$\ef(F)=\bigcup_{\substack{Q\in\cd_* : Q\subseteq F \text{ but } \\ \text{$Q\subseteq F'$ for no $F'\in\chf(F)$}}} Q.$$ Therefore,
$$
\abs{\be_{\cd_*}^\mu f}_E1_{\ef(F)}\leq  \sum_{\substack{Q\in\cd_* : Q\subseteq F \text{ but } \\ \text{$Q\subseteq F'$ for no $F'\in\chf(F)$}}} \angles{\abs{f}_E}^\mu_Q 1_Q \leq 2 \angles{\abs{f}_E}^\mu_F.
$$

Next, we check the inequality \eqref{propertysecond}. On the one hand, $\angles{\abs{f}_E}_{F'}^\mu\leq \frac{\mu(\hat{F'})}{\mu(F')} \angles{\abs{f}_E}_{\hat{F'}}^\mu$, and, on the other hand, by the stopping condition, $\angles{\abs{f}_E}^\mu_{\hat{F'}} \leq 2 \angles{\abs{f}_E}^\mu_F$; combining these estimates yields the inequality \eqref{propertysecond}.

Next, we note that the inequality \eqref{propertythird} follows from Lemma \ref{elegantestimate} together with the stopping condition:
\begin{equation*}
\begin{split}
\norm{\sum_{F'\in\chf(F)} \int_{F'} f \dmu \frac{1_{\hat{F'}}}{\mu(\hat{F'})}}_{L^t_E(\mu)}&\lesssim_t \big(\sup_{F'\in\chf(F)}\, \angles{\abs{f}_E}_{\hat{F'}}^\mu\big)^{1/t'} \big(\int_{\bigcup_{F'} F'} \abs{ f}_E \dmu \big)^{1/t}\\
&\leq 2^{1/t'} \angles{\abs{f}_E}_{F}^\mu \mu(F)^{1/t}.
\end{split}
\end{equation*}

Finally, we check the replacement rule \eqref{eq_principal_replacament}. Assume that $\pif(Q)=F$. We write
$$
\int_Q f \dmu =\sum_{F'\in\chf(F)}  \int_Q f 1_{F'} \dmu +\int_Q 1_{\ef(F)} f \dmu.
$$ Assume that $Q$ and $F'$ are such that $F'\cap Q\neq \emptyset$. Then, by dyadic nestedness, either $F'\subsetneq Q$ or $Q\subseteq F'$, the latter of which is excluded by the condition $\pif(Q)=F$. Therefore, $F'\subsetneq Q$ (and, hence, $\hat{F'}\subseteq Q$). Now,
$$
\int_Q f 1_{F'} \dmu = \int_{F'} f \dmu =\int_Q \angles{f}^\mu_{F'} 1_{F'} \dmu,$$
and
$$\int_Q f 1_{F'} \dmu= \int_Q  \big(\int_{F'} f \dmu\big) \frac{1_{\hat{F'}}}{\mu(\hat{F'})}\dmu.
$$
\end{proof}
\begin{remark}We note that if the collection $\cd$ is such that it contains cubes $Q\in\cd$ shrinking to almost every point $x\in\ef(F)$, then, by the Lebesgue differentiation theorem,  $$\abs{f}_E1_{\ef(F)} =\lim_{N\to \infty} \abs{\be^\mu _{\{Q\in\cd: \ell(Q)\geq 2^{-N}\}}f}_E1_{\ef(F)} \leq 2 \angles{\abs{f}_E}^\mu_F.$$ The finest averaging operator $\be_{\cd_*}^\mu$ appears in the lemma because we assume that the collection $\cd$ is finite (and, therefore, has no shrinking cubes). 

This appearance is harmless when we are considering quantities that only take into account the finest averaging: For example, $\norm{\be_{\cd_*}^\mu f}_{L^p_E(\mu)}\leq \norm{f}_{L^p_E(\mu)}$, and, whenever $\cd$ is a truncated dyadic system,
 $$\mbar^\mu_\cd f=\mbar^\mu_\cd (\be_{\cd_*}^\mu f),\quad\text{and}\quad T_{\cd}(f\mu)= T_{\cd}((\be_{\cd_*}^\mu f)\mu).
$$
Observe that, in the definition $
T_{\cd}(f\mu):=\sum_{Q\in\cd}\lambda_Q \int_Q f \dmu 1_Q,
$
we may assume that $\cd$ is a truncated dyadic system (by including some zero coefficients $\lambda_Q$, if necessary).

\end{remark}

\begin{lemma}[Lemma 3.3 in \cite{lopezsanchez2012}, by L\'opez-S\'anchez, Martell, and Parcet]\label{elegantestimate}Let $1\leq p<\infty$. Let $\mu$ be a locally finite Borel measure. Let  $h$ be a non-negative real-valued function. Let $\{R\}$ be a collection of pairwise disjoint dyadic cubes. Then
$$
\norm{\sum_{R} \int_{R} h \dmu \frac{1_{\hat{R}}}{\mu(\hat{R})}}_{L^p(\mu)}\lesssim_p \big(\sup_{R}\, \angles{h}_{\hat{R}}^\mu\big)^{1/p'} \big(\int_{\bigcup_{R} R} h \dmu \big)^{1/p}.
$$
\end{lemma}

\section{Corollaries}\label{sec_corollaries}
In this section, we state some corollaries of the characterization of the boundedness of the operator $T_\lambda(\cdotroomy\mu):L^p_C(\mu)\to L^p_D(\mu)$ by the dual pairing testing condition \eqref{eq_dualpairingtestingcondition}, or, equivalently, by the endpoint testing condition \eqref{eq_limitingcaseoflinfinitydirect}. 

First, Theorem \ref{thm_dualpairingtesting} provides an alternative proof for the following well-known John--Nirenberg-type inequality:
\begin{corollary}[John--Nirenberg-type inequality]Let $\mu$ be a locally finite Borel measure. Let $\{\lambda_Q\}_{Q\in\cd}$ be non-negative real numbers. Then, for each $1<p<\infty$, we have
$$
\sup_{R\in\cd} \frac{1}{\mu(R)}\norm{\sum_{\substack{Q\in\cd:\\ Q\subseteq R}} \lambda_Q 1_Q}_{L^1(\mu)}\eqsim_p \sup_{R\in\cd} \frac{1}{\mu(R)^{1/p}}\norm{\sum_{\substack{Q\in\cd:\\ Q\subseteq R}} \lambda_Q 1_Q}_{L^p(\mu)}.
$$
\end{corollary}
\begin{proof}The equivalence follows from observing that the left-hand side of the inequality is the end-point direct $L^\infty$ testing constant \eqref{eq_limitingcaseoflinfinitydirect} and the right-hand side is the direct $L^\infty$ testing constant \eqref{eq_testing_direct}   for the operator $T(\cdotroomy\mu):L^p(\mu)\to L^p(\mu)$ defined by $T(f\mu):=\sum_{Q\in\cd} \lambda_Q \angles{f}^\mu_Q 1_Q$.
\end{proof}

The next embedding theorem was proven by Nazarov, Treil, and Volberg \cite{nazarov2003} by using the Bellman function method; an alternative proof for this theorem is provided by Theorem \ref{thm_dualpairingtesting}.
\begin{corollary}[Embedding theorem, Theorem 3.1 in \cite{nazarov2003}]Let $\mu$ be a locally finite Borel measure. Let $\{\beta_Q\}$ be non-negative real numbers. Let $T(\,\cdot\, \mu)$ be defined by $T(f \mu):=\{\angles{f}^\mu_Q 1_Q\}_{Q\in\cd}$, so that $$
\abs{T(f\mu)}_{\ell^s(\cd,\beta)}:=\Big(\sum_{Q\in\cd} \beta_Q (\angles{f}^\mu_Q 1_Q)^s \Big)^{1/s}.
$$
Then, the following assertions are equivalent:
\begin{itemize}
\item[i)] $T(\,\cdot\,\mu):L^p(\mu) \to L^p_{\ell^s(\cd,\beta)}(\mu)$ is bounded for all $1<p,s<\infty$.
\item[ii)] $T(\,\cdot\,\mu):L^{p_0}(\mu) \to L^{p_0}_{\ell^{s_0}(\cd,\beta)}(\mu)$ is bounded for some $1<p_0,s_0<\infty$.
\item[iii)] The direct testing constant
$$
\mathfrak{T}^{p_0}_{s_0}:=\sup_{R\in\cd} \frac{\norm{T_R(1\mu)}_{L^{p_0}_{\ell^{s_0}(\cd,\beta)}}}{\mu(R)^{1/{p_0}}}
$$
is finite for some $1<p_0,s_0<\infty$.
\item[iv)] The Carleson constant
$$
\mathfrak{C}:=\sup_{R\in\cd} \frac{1}{\mu(R)}\sum_{\substack{Q\in\cd: \\Q\subseteq R}} \beta_Q \mu(Q)
$$
is finite.
\end{itemize}
Quantitatively, we have:
$$
\norm{T(\,\cdot\,\mu)}_{L^p(\mu) \to L^p_{\ell^s(\cd,\beta)}(\mu)}^s\lesssim_{p,s} \mathfrak{C}\lesssim_{s_0} (\mathfrak{T}^{p_0}_{s_0})^{s_0}\leq \norm{T(\,\cdot\,\mu)}_{L^{s_0}(\mu) \to L^{s_0}_{\ell^{s_0}(\cd,\beta)}(\mu)}^{s_0}.
$$
\end{corollary}
\begin{proof}We observe that $\mathfrak{T}^{s}_{s}=\mathfrak{C}^{1/{s}}$ for every $s\in(1,\infty)$. First,  we prove that iii) implies iv) via the dual pairing testing. By H\"older's inequality, the direct testing condition implies the dual pairing testing condition:
\begin{equation*}
\begin{split}
\mathfrak{P}_{s_0}&:=\sup_{R\in\cd}\sup_{\substack{f\in L^\infty(R,\mu),\\ g\in L^\infty_{\ell^{s_0'}(\cd,\beta)}(R,\mu), }}\frac{\abs{\int g T_R (f\mu) \dmu}}{\norm{g}_{L^\infty_{\ell^{s_0'}(\cd,\beta)}(R,\mu)} \norm{f}_{L^\infty(R,\mu)} \mu(R)}\\
& \leq  \sup_{R\in\cd}\frac{\norm{T_R(1_R\, \mu)}_{L^{p_0}_{\ell^{s_0}(\cd,\beta)}}}{\mu(R)^{1/{p_0}}}=:\mathfrak{T}^{p_0}_{s_0}.
\end{split}
\end{equation*}
Hence, by Theorem \ref{thm_dualpairingtesting}, we have $
\norm{T(\,\cdot\,\mu)}_{L^p(\mu) \to L^p_{\ell^{s_0}(\cd,\beta)}(\mu)}\lesssim_{p,{s_0}} \mathfrak{T}^{p_0}_{s_0}
$ for every $p\in(1,\infty)$, which in particular (for $p=s_0$) implies that
$$
\mathfrak{C}^{1/{s_0}}=\mathfrak{T}^{s_0}_{s_0}\leq \norm{T(\,\cdot\,\mu)}_{L^{s_0}(\mu) \to L^{s_0}_{\ell^{s_0}(\cd,\beta)}(\mu)}\lesssim_{s_0} \mathfrak{T}^{p_0}_{s_0}.
$$

Next, we prove that iv) implies i) via the dual pairing testing condition. Again, by H\"older's inequality, for every $s\in(1,\infty)$, we have
$
\mathfrak{P}_s\leq \mathfrak{T}^s_s=\mathfrak{C}^{1/s}.
$
Hence, by Theorem \ref{thm_dualpairingtesting}, $
\norm{T(\,\cdot\,\mu)}_{L^p(\mu) \to L^p_{\ell^s(\cd,\beta)}(\mu)}\lesssim_{p,s} \mathfrak{C}^{1/s}
$ for every $p,s\in(1,\infty)$.

\end{proof}

Finally, Theorem \ref{thm_dualpairingtesting} provides an extension of the dyadic Carleson embedding theorem for the class of matrices whose all entries are non-negative:
\begin{corollary}[$L^\infty$ version of the Carleson embedding theorem for matrices with non-negative entries]\label{thm_linfinitycarlesonembedding}Let $\mu$ be a locally finite Borel measure. Let $\{\lambda_Q\}_{Q\in\cd}$ be such that each $\lambda_Q:\ell^2 \to \ell^2$ is a symmetric (infinite dimensional)  matrix whose all entries are non-negative. Then
\begin{equation}
\label{eq_linfinitycarlesonvariant}
\sup_{f\in L^2_{\ell^2}(\mu)} \frac{\sum_{Q\in\cd}(\angles{f}_Q^\mu)^t\lambda_Q \angles{f}^\mu_Q}{\norm{f}_{L^2_{\ell^2}(\mu)}^2}\eqsim \sup_{R\in\cd} \sup_{f\in L^\infty_{\ell^2}(R,\mu)} \frac{\sum_{Q\in\cd: Q\subseteq R}(\angles{f}_Q^\mu)^t\lambda_Q \angles{f}_Q^\mu }{\norm{f}_{L^\infty_{\ell^2}(R,\mu)}^2 \mu(R)}.
\end{equation}
\end{corollary}
\begin{proof}A well-known trick of {\it depolarisation} can be phrased as follows: Let $(V,\norm{\cdotroomy}_V)$ be a normed vector space, and let $B(\cdotroomy,\cdotroomy):V\times V\to \br$ be a symmetric bilinear form. Assume that $B(v,v)\lesssim \norm{v}_V^2$ for all $v\in V$. Then $B(v,v')\lesssim \norm{v}_V \norm{v'}_V$ for all $v,v'\in V$. From this trick, it follows that
\begin{equation}\label{eq_intermediatestep}
\sup_{R\in\cd} \;\sup_{f\in L^\infty_{\ell^2}(R,\mu), g\in L^\infty_{\ell^2}(R,\mu)} \frac{\sum_{Q\in\cd: Q\subseteq R}(\angles{f}_Q^\mu)^t\lambda_Q^\mu \angles{g}_Q }{\norm{f}_{L^\infty_{\ell^2}(R,\mu)}\norm{g}_{L^\infty_{\ell^2}(R,\mu)}  \mu(R)}\lesssim  \text{R.H.S}\eqref{eq_linfinitycarlesonvariant} .
\end{equation}
The left-hand side of the equation \eqref{eq_intermediatestep} is the dual pairing testing constant for the dual norm inequality $\sum_{Q\in\cd}\lambda_Q (\angles{f}^\mu_Q)^t \lambda_Q \angles{g}^\mu_Q\lesssim \norm{f}_{L^2_{\ell^2}(\mu)}\norm{g}_{L^2_{\ell^2}(\mu)}$. 
\end{proof}

\section{Questions about the borderline of the vector-valued testing conditions}
The questions are posed in the unweighted case since the answers are unknown even in this case. The first question is about weakening the type of the testing condition in the characterization. The operator $T(\cdotroomy \mu):L^p_E(\mu) \to L^p_E(\mu)$ satisfies {\it the constant function testing condition} if
\begin{equation}
\label{eq_testing_constant}
\norm{T_R(e 1_R \mu)}_{L^p_E(\mu)}\leq \mathfrak{S} \abs{e}_E \mu(R)^{1/p}
\end{equation}
for every $R\in\cd$, and every $e\in E$.
This testing condition is weaker than the direct $L^\infty$ testing conditions \eqref{eq_linftytestingconditions} in that $\mathfrak{S}\leq \mathfrak{T}$.
Note that,  in the real-valued case, this testing condition and the $L^\infty$ testing condition  both coincide with the Sawyer testing condition \eqref{eq_sawyer testing}.

\begin{question}[Borderline case: Can we use the testing condition \eqref{eq_testing_constant} in Theorem \ref{thm_twoweight} in place of the $L^\infty$ testing condition \eqref{eq_linftytestingconditions}?] In particular, contrasting with Theorem \ref{thm_twoweight}, is it true that there exists a constant $C$ such that
\begin{equation*}\label{eq_directconstanttesting}
\sup_{f\in L^2_{\ell^2}} \frac{\norm{\sum_{Q\in\cd} {\lambda_Q} \angles{f}_Q1_Q }_{L^2_{\ell^2}}}{\norm{f}_{L^2_{\ell^2}}} \leq C \sup_{R\in\cd} \sup_{a\in\ell^2} \frac{\norm{\big(\sum_{Q\in\cd: Q\subseteq R} \lambda_Q1_Q\big) a}_{L^2_{\ell^2}}}{\abs{a}_{\ell^2} \abs{R}^{1/2}}
\end{equation*}
for all $\{\lambda_Q\}_{Q\in\cd}$ such that each $\lambda_Q:\ell^2\to\ell^2$ is a symmetric matrix whose all entries  are non-negative? Or, contrasting with Theorem \ref{thm_dualpairingtesting}, is it true that there exists a constant $C$ such that
\begin{equation}\label{eq_carlesonconstanttesting}
\sup_{f\in L^2_{\ell^2}} \frac{\sum_{Q\in\cd}\angles{f}_Q^t\lambda_Q \angles{f}_Q}{\norm{f}_{L^2_{\ell^2}}^2}\leq C \sup_{R\in\cd} \frac{\norm{\sum_{Q\in\cd: Q\subseteq R} \lambda_Q }_{\ell^2\to \ell^2}}{\abs{R}}
\end{equation}
for all $\{\lambda_Q\}_{Q\in\cd}$ such that each $\lambda_Q:\ell^2\to\ell^2$ is a symmetric matrix whose all entries  are non-negative?
\end{question}
\begin{remark}We note that Nazarov, Treil, and Volberg \cite{nazarov1997} proved that the estimate \eqref{eq_carlesonconstanttesting} fails for a different class of matrices: the class of positive-semi-definite matrices. Recall that a symmetric matrix $M$ is {\it positive-semi-definite} if $x^tMx\geq 0$ for all column vectors $x$. \end{remark}

In our characterizations, the assumption that the Banach space has the Hardy--Littlewood property can be replaced by assuming that the measure is doubling, or by strenghtening the testing condition (see Theorem \ref{thm_alternativetesting}). The second question is about omitting every additional assumption.
\begin{question}[Borderline case: Can we omit every additional assumptions in Theorem \ref{thm_alternativetesting}?]\label{question_additionalassumption}Let $p\in(1,\infty)$.  Let $(E,\abs{\cdotroomy}_E,\leq)$ be a Banach lattice. Let $\mu$ be a locally finite Borel measure. Then, is is true that the operator $T_\lambda(\cdotroomy):L^p_E(\mu)\to L^p_E(\mu)$ is bounded if and only if it satisfies the direct $L^\infty$ testing condition \eqref{eq_testing_direct}?
\end{question}

\appendix 
\section{On the dyadic lattice Hardy--Littlewood maximal operator}
\subsection{Dyadic and the centered lattice maximal function are comparable}\label{sec_comparisionofdyadicandcentered}
The dyadic Hardy--Littlewood maximal function $\bar{M}_\cd f$ is defined by
$$
\bar{M}_\cd f(x):=\sup_{Q\in\cd: Q\ni x } \angles{f}_Q,
$$
where $\cd$ is a collection of dyadic cubes, and the centered lattice Hardy--Littlewood maximal function $\bar{M}_J$ is defined by
$$
\bar{M}_J f(x):=\sup_{r\in J}\; \angles{f}_{B(x,r)},
$$ 
where $J$ is a finite set of radii. For the Lebesgue measure,  these maximal functions are pointwise comparable in the lattice order: For each finite collection $\cd$ of dyadic cubes, there exists a finite set $J$ of radii  such that $\bar{M}_\cd f(x)\leq \bar{M}_{J_\cd}f(x)$ for every $x\in \br^d$. Conversely, for each finite set $J$ of radii, there exist collections $\cd^\alpha_J$ of (shifted) dyadic cubes such that $ \bar{M}_{J}f(x)\leq \sum_{\alpha} \bar{M}_{\cd^\alpha_J} f(x)
$ for every $x\in\br^d$.

This comparision follows from the following well-known observation: For each dyadic cube $Q\in\cd$, there exists a ball $B$ such that $Q\subseteq B$ and $\abs{Q}\eqsim \abs{B}$. Conversely, for each ball $B$, there exists a dyadic cube $Q$ in some shifted dyadic system $\cd^\alpha$ such that $B\subseteq Q$ and $\abs{B}\eqsim \abs{Q}$. For a proof, see, for example, \cite[Lemma 2.5]{hytonen2013d}. Recall that, for each $\alpha\in\{0,\frac{1}{3}\}^d$, {\it the shifted dyadic system}  $\cd^\alpha$ on $\br^d$ is defined by $$\cd^\alpha:=\{2^{-k}([0,1)^d+(-1)^k \alpha + j) : k\in\bz, j\in\bz^d\}.$$

\subsection{Universal norm bound}\label{sec_universalnormbound}
The universal bound for the lattice maximal operator,
$$\norm{\bar{M}^\mu}_{L^p_E(\br^d,\mu) \to L^p_E(\br^d,\mu)}\lesssim_p  \norm{\bar{M} }_{L^p_E(\br^d) \to L^p_E(\br^d)},$$ 
 follows from either of the following techniques: 
 \begin{itemize}
 \item The boundedness of the dyadic real-valued maximal function is characterized by means of the existence of a Bellman function, by Nazarov and Treil \cite[Section 1]{nazarov1996}. This characterization works also for the dyadic lattice maximal function. 
 \item In the spirit of Burkolder's \cite{burkholder1991} characterization of the boundedness of the martingale transform, the boundedness of the martingale  Rademacher maximal function is characterized by means of the existence of an auxiliary function with certain boundedness and concavity properties,  by Kemppainen \cite[Section 7]{kemppainen2011}. This characterization works also for the dyadic lattice maximal function, once the Rademacher bound is replaced by the lattice supremum. This together with an unpublished manuscript containing the proof was communicated to the author by Kemppainen.  
 \end{itemize}
 For reader's convenience,  we represent a proof for the universal bound. The universal bound follows from Proposition \ref{prop_boundednessimpliesbellman} and Proposition \ref{prop_existenceimpliesboundedness} together with the observation that 
$$\norm{\bar{M} }_{L^p_E(\br) \to L^p_E(\br)}\leq \norm{\bar{M} }_{L^p_E(\br^d) \to L^p_E(\br^d)}.$$ 
These propositions follow from Nazarov and Treil's \cite[Section 1]{nazarov1996} Bellman function technique.

\begin{proposition}[Boundedness implies the existence of a Bellman function, \cite{nazarov1996}]\label{prop_boundednessimpliesbellman}Let $(E,\abs{\cdotroomy},\leq)$ be a Banach lattice. Assume that there exists a constant $\mathfrak{B}$ such that 
$$
\norm{\bar{M}_\cd }_{L^p_E(\br) \to L^p_E(\br)}\leq \mathfrak{B}
$$
for all finite collections $\cd$ of dyadic intervals. Then, there exists a Bellman function $B(f,F,L):E_+\times \br_+ \times E_+\to \br_+$ that  has the following properties:
\begin{itemize}
\item[i)] (Boundedness from below) $\abs{L}_E^P\leq B(f,F,L)$ whenever $0<\abs{f}_E^p\leq F$, or $f=0$ and $F=0$.
\item[ii)] (Boundedness from above) $B(f,F,L)\lesssim_p \mathfrak{B}^p(F+\abs{L}_E^p)$.
\item[ii)] (Invariance) $B(f,F,L)=B(f,F,\sup\{L,f\})$.
\item[iv)] (Concavity) For each $L\in E$, the function $(f,F)\mapsto B(f,F,L)$ is midpoint concave.
\end{itemize}

\end{proposition}
\begin{remark}Since every midpoint concave function that is locally bounded from below is concave (for a proof, see, for example, \cite[Section 7]{kemppainen2011}), the function $(f,F)\mapsto B(f,F,L)$  is in fact concave.
\end{remark}
\begin{proof}[Proof from \cite{nazarov1996}] For each $I\in\cd$, the function $B_I(f,F,L):E_+\times \br_+ \times E_+\to \br_+$ is defined by
\begin{equation}\label{bellman_nazarov}
\begin{split}
B_I(f,F,L):=\sup \Big\{&\frac{1}{\abs{I}}\int_I \abs{\sup\{\sup_{\substack{J:J\subseteq I \\ \ell(J)\geq 2^{-N} \ell(I)}} \angles{\phi}_{J}1_J, L \}}_E^p\dx : \\
&: \text{$\phi_I:\br^d\to E_+$ is locally integrable and satisfies }\\
&\phantom{:} \text{ $\angles{\phi_I}_I=f$ and $\angles{\abs{\phi_I}_E^p}_I=F,$ $N\in\bn$}\Big\}.
\end{split}
\end{equation}By self-similarity of the dyadic intervals, the function $B_I$ does not depend on the interval $I$ and can be denoted by $B$. This Bellman function is introduced by Nazarov and Treil \cite[Section 1]{nazarov1996}. In the real-valued case (that is, $E=\br$), it is explicitly computed by Melas \cite[Theorem 1]{melas2005}.

Next, we check the properties for the Bellman function $B$. The boundedness from below holds because
for each $f\in E_+$ and $F\in\br_+$ such that $0<\abs{f}^p_E\leq F$ there exists $\phi:\br^d\to E_+$ such that $\angles{\phi}_I=f$ and $\angles{\abs{\phi}_E^p}=F$. The boundedness from above follows from the assumed norm estimate. The invariance follows from observing that, under the constraint $\angles{\phi}_I=f$, both the vector $f$ and the vector $L$ belong to the set $\{\angles{\phi}_{J}1_J,L \}_{\substack{J:J\subseteq I,\\ \ell(J)\geq 2^{-N}} \ell(I)}$ of which the lattice supremum is taken. 

Finally, we check the midpoint concavity. Let $I_-$ and $I_+$ be the dyadic children of $I$. Let $\phi_{I_-}$ be such that $\angles{\phi_{I_-}}_{I_-}=f_-$ and $\angles{\abs{\phi_{I_-}}_E^p}_{I_-}=F_-$, and, similarly, $\phi_{I_+}$ be such that $\angles{\phi_{I_+}}_{I_+}=f_+$ and $\angles{\abs{\phi_{I_+}}_E^p}_{I_+}=F_+$. Now, the function $\phi_I:=\phi_{I_-}+\phi_{I_+}$ satisfies $f:=\angles{\phi_I}_I=\frac{1}{2}(\angles{\phi_{I_-}}_{I_-}+ \angles{\phi_{I_+}}_{I_+})=\frac{1}{2}(f_-+f_+)$, and $F:=\angles{\abs{\phi_I}_E^p}_I=\frac{1}{2}(\angles{\abs{\phi_{I_-}}_E^p}_{I_-}+ \angles{\abs{\phi_{I_+}}_E^p}_{I_+})=\frac{1}{2}(F_-+F_+)$. We estimate
\begin{equation*}
\begin{split}
&\frac{1}{2}\frac{1}{\abs{I_-}}\int_{I_-} \abs{\sup\{\sup_{\substack{J:J\subseteq {I_-} \\ \ell(J)\geq 2^{-N} \ell(I_-)}} \angles{\phi_{I_-}}_{J}1_J, L \}}_E^p\dx\\
&+\frac{1}{2}\frac{1}{\abs{I_+}}\int_{I_+} \abs{\sup\{\sup_{\substack{J':J'\subseteq {I_+} \\ \ell(J')\geq 2^{-N'} \ell(I_+)}} \angles{\phi_{I_+}}_{J}1_J, L \}}_E^p\dx\\
&=\frac{1}{\abs{I}}\int_I \abs{\sup\{\sup_{\substack{J,J':J\subseteq {I_-}, J'\subseteq I_+ \\ \ell(J)\geq 2^{-N} \ell(I_-),\ell(J)\geq 2^{-N'} \ell(I_+) }} \angles{\phi_{I}}_{J}1_J, L \}}_E^p\dx\\
&\leq \frac{1}{\abs{I}}\int_I \abs{\sup\{\sup_{\substack{J:J\subseteq {I}, \\ \ell(J)\geq 2^{-(\max\{N,N'\}+1)}\ell(I)}}   \angles{\phi_{I}}_{J}1_J, L \}}_E^p\dx\\
&\leq B(f,F,L),
\end{split}
\end{equation*}
from which the midpoint concavity follows by taking the suprema.
\end{proof}

\begin{remark}An alternative Bellman function can be defined as follows. For each $I\in\cd$, the function $\tilde{B}_I(f,F,A):E_+\times \br_+ \times \{A\subseteq E_+ : \text{ $A$ finite}\}\to \br_+$ is defined by
\begin{equation}\label{bellman_kemppainen}
\begin{split}
\tilde{B}_I(f,F,A):=\sup \Big\{&\frac{1}{\abs{I}}\int_I \abs{\sup\big( A\cup \{ \angles{\phi}_{J}1_J \}_{\substack{J:J\subseteq I \\ \ell(J)\geq 2^{-N} \ell(I)}} \big)}_E^p\dx : \\
&: \text{$\phi_I:\br^d\to E_+$ is locally integrable and satisfies }\\
&\phantom{:} \text{ $\angles{\phi_I}_I=f$ and $\angles{\abs{\phi_I}_E^p}_I=F,$ $N\in\bn$}\Big\}.
\end{split}
\end{equation}
Again, by self-similarity of the dyadic intervals, the function $\tilde{B}_I$ does not depend on the dyadic interval $I$. Hence, it can be denoted by $\tilde{B}$. The function $\tilde{B}(f,F,A)$ has the following properties:
\begin{itemize}
\item[i')] (Boundedness from below) $\abs{\sup A}_E^P\leq \tilde{B}(f,F,L)$ whenever $0<\abs{f}_E^p\leq F$, or $f=0$ and $F=0$.
\item[ii')] (Boundedness from above) $\tilde{B}(f,F,A)\lesssim_p \mathfrak{B}^p(F+\abs{\sup A}_E^p)$
\item[iii')] (Invariance) $\tilde{B}(f,F,A)=\tilde{B}(f,F,A\cup\{f\})$
\item[iv')] (Concavity) For each finite $A\subseteq E_+$, the function $(f,F)\mapsto \tilde{B}(f,F,A)$ is midpoint concave,
\end{itemize}
By considering the Rademacher bound $\mathcal{R}(A)$ in place of the lattice supremum $\sup A$, the Bellman function $\tilde{B}(f,F,A)$  can be viewed as a variant of the auxiliary function that was introduced by Kemppainen \cite[Proposition 7.1]{kemppainen2011} to characterize the boundedness of the Rademacher maximal function $\mathcal{R}_{Q\in\cd} \angles{f}_Q 1_Q $. 

We remark that, in the case of the lattice supremum, the function $\tilde{B}(f,F,A)$ defined in \eqref{bellman_kemppainen} reduces to the Bellman function $B(f,F,L)$ defined in \eqref{bellman_nazarov} by using the identity $\tilde{B}(f,F,A)=B(f,F,\sup A)$, whereas, in the case of the Rademacher bound, there is no such a reduction. This is because the reduction is based on the identity 
$\sup \{A\cup B\}=\sup\{\sup A,\sup B\}$ for the lattice supremum, 
whereas there is no analogous identity for the Rademacher bound.
\end{remark}

\begin{proposition}[Existence of a Bellman function implies the boundedness,  \cite{nazarov1996}]\label{prop_existenceimpliesboundedness}Let $(E,\abs{\cdotroomy}_E,\leq)$ be a Banach lattice. Assume that $\tilde{B}(f,F,A):E_+\times \br_+ \times \{A\subseteq E_+ : \text{ $A$ finite}\}\to \br_+$ is a function having the above-mentioned properties. Then $$
\norm{\bar{M}^\mu_\cd}_{L^p_E(\br^d,\mu)\to L^p_E(\br^d,\mu)}\lesssim_p \mathfrak{B}
$$
for all finite collections $\cd$ of dyadic intervals and all locally finite Borel measures $\mu$.
\end{proposition}
\begin{proof}[Proof by a slight adaptation of \cite{nazarov1996} in the spirit of \cite{kemppainen2011}] Let $\mu$ be a locally finite Borel measure. Let $Q$ be a dyadic cube and let $Q'\in\ch_\cd(Q)$ be its dyadic children. Let $f:\br^d\to E_+$ be a locally integrable function. Note that $$\angles{f}_Q^\mu=\sum_{Q'\in\ch_\cd(Q)}\frac{\mu(Q')}{\mu(Q)}\angles{f}_{Q'}^\mu, \quad \text{ and }\quad \angles{\abs{f}_E^p}_Q^\mu=\sum_{Q'\in\ch_\cd(Q)} \frac{\mu(Q')}{\mu(Q)}\angles{\abs{f}_E^p}_{Q'}^\mu.$$ Since every every mid-point concave function that is locally bounded from below is in fact concave, the function $(f,F)\mapsto \tilde{B}(f,F,A)$ is in fact concave. From the properties of the Bellman function, it follows that
\begin{equation}\label{bellman_iteration}
\begin{split}
&\sum_{Q'\in\ch_\cd(Q)}\mu(Q') \tilde{B}(\angles{f}_{Q'}^\mu,\angles{\abs{f}_E^p}_{Q'}^\mu, \{\angles{f}^\mu_R\}_{R:R\supseteq Q'})\\
&\overset{(iii')}{=} \sum_{Q'\in\ch_\cd(Q)}\mu(Q') \tilde{B}(\angles{f}_{Q'}^\mu,\angles{\abs{f}_E^p}_{Q'}^\mu, \{\angles{f}^\mu_R\}_{R:R\supseteq Q}) \\
& \overset{(iv')}{\leq} \mu(Q) \tilde{B}(\angles{f}^\mu_Q,\angles{\abs{f}_E^p}_Q^\mu, \{\angles{f}^\mu_R\}_{R:R\supseteq Q}).
\end{split}
\end{equation}
Fix a dyadic cube $Q_0$ and a non-negative integer $N$. Iterating the inequality \eqref{bellman_iteration} and using the properties of the Bellman function yields
\begin{equation*}
\begin{split}
&\int \abs{\sup_{R:R\subseteq Q_0,\ell(R)\geq 2^{-N} \ell(Q_0)} \angles{f}^\mu_R 1_R} \dmu\\
&=\sum_{\substack{Q:Q\subseteq Q_0,\\ \ell(Q)=2^{-N}\ell(Q_0)}} \mu(Q) \abs{\sup_{R:Q_0\supseteq R\supseteq Q}\angles{f}_R^\mu}^p_E\\
&\overset{(i')}{\leq}  \sum_{\substack{Q:Q\subseteq Q_0,\\ \ell(Q)=2^{-N}\ell(Q_0)}} \mu(Q) \tilde{B}(\angles{f}^\mu_Q, \angles{\abs{f}_E^p}_Q^\mu, \{\angles{f}_R^\mu\}_{R:Q_0\supseteq R\supseteq Q})\\
&\overset{\eqref{bellman_iteration}}{\leq }\sum_{\substack{Q:Q\subseteq Q_0,\\ \ell(Q)=2^{-{(N-1)}}\ell(Q_0)}} \mu(Q) \tilde{B}(\angles{f}^\mu_Q, \angles{\abs{f}_E^p}_Q^\mu, \{\angles{f}_R^\mu\}_{R:Q_0\supseteq R\supseteq Q})\\
&\leq\cdots\leq \mu(Q_0) \tilde{B}(\angles{f}^\mu_{Q_0}, \angles{\abs{f}_E^p}_{Q_0}^\mu, \{\angles{f}_{Q_0}^\mu\})\\
&\overset{(ii')}{\lesssim_p}   \mathfrak{B}\mu(Q_0) \angles{\abs{f}_E^p}_{Q_0}^\mu +\abs{\angles{f}_{Q_0}^\mu}_E^p\leq 2 \mathfrak{B} \int_{Q_0} \abs{f}_E^p \dmu.
\end{split}
\end{equation*}
\end{proof}

\subsection{Endpoint $L^\infty$ testing condition}\label{sec_alternativeproofformaximal}
A collection $\cd$ of dyadic cubes is {\it a truncated dyadic system} if $$\cd=\{Q : Q\subseteq Q_0, \ell(Q)\geq 2^{-N} \ell(Q_0)\}=:\cd^{Q_0}_N$$
for some dyadic cube $Q_0$ and some positive integer $N$. For each $R\in\cd$, the {\it localized} dyadic lattice Hardy--Littlewood operator $\bar{M}_{\cd,R}$ is defined by
$$
\bar{M}_{\cd,R}f:=\sup_{\substack{Q\in\cd:\\Q\subseteq R}} \angles{f}_Q 1_Q.
$$
\begin{theorem}[Boundedness of the dyadic lattice maximal operator is characterized by the endpoint direct $L^\infty$ testing condition, \cite{torrea1993}] \label{prop_maximalcharacterization}Let $1<p<\infty$. Let $\cd$ be a truncated dyadic system on $\br^d$. Then
$$
\norm{\mbar_\cd}_{L^p_E\to L^p_E}\eqsim_{p,d} \mathfrak{M},
$$
where the endpoint $L^\infty$ testing constant $\mathfrak{M}$ is the least constant such that
\begin{equation}\label{eq_testingmaximal}
\norm{\bar{M}_{\cd,R} f}_{L^1_E}\leq \mathfrak{M}\norm{f}_{L^\infty_E(R)} \abs{R}
\end{equation}
for every $R\in\cd$, and every $f\in L^\infty_E(R)$.
\end{theorem}
This theorem was proven Garc\'ia-Cuerva, Mac\'ias, and Torrea  \cite{torrea1993} by applying the theory of vector-valued singular integrals to a smooth, linearized version of the lattice maximal function. Here, we give an alternative proof by using stopping cubes.
\begin{proof}[Alternative proof by stopping cubes]
Let $\cf$ be the stopping family defined by the following stopping children: For each $F\in\cf$, the children $\chf(F)$ are the maximal dyadic cubes $F'\subseteq F$ such that
\begin{equation}\label{max_stoppingcondition_lattice}
\abs{\sup_{\substack{Q\in \cd: Q\supseteq F'}}\angles{f}_Q }_E\geq 4 \angles{\abs{\bar{M}_\cd f}_E}_F
\end{equation}
or
\begin{equation}\label{max_stoppingcondition_average}
\angles{\abs{f}_E}_{F'}>4 \angles{\abs{f}_E}_F.
\end{equation}
The stopping collection $\cf$ is sparse because
$$
\sum_{F'} \abs{F'}\leq \sum_{\substack{\text{$F'$ chosen by}\\\text{the first condition}}} \abs{F'} + \sum_{\substack{\text{$F'$ chosen by}\\\text{the second condition}}}\abs{F'}\leq( \frac{1}{4}+\frac{1}{4})\abs{F}=\frac{1}{2}\abs{F}.
$$ 
By arranging the dyadic cubes according to the stopping parents, using the $L^p$ variant of Pythagoras' theorem (Lemma \ref{lem_pythagoras}), and pulling out the $L^\infty_E$ norm,
\begin{equation*}
\begin{split}
\norm{\mbar_\cd f}_{L^p_E}&=\norm{\sup_{F\in F} \sup_{\substack{Q\in \cd:\\ \pif(Q)=F}} \angles{f}_Q 1_Q }_{L^p_E}\\
&\leq \norm{\sum_{F\in\cf}\sup_{\substack{Q\in \cd:\\ \pif(Q)=F}} \angles{f}_Q 1_Q }_{L^p_E}\\
&\lesssim_p  \Big(\sum_{F\in\cf}\norm{\sup_{\substack{Q\in \cd:\\ \pif(Q)=F}} \angles{f}_Q 1_Q }_{L^p_E}^p\Big)^{1/p}\\
&\leq \Big(\sum_{F\in\cf}\norm{\sup_{\substack{Q\in \cd:\\ \pif(Q)=F}} \angles{f}_Q 1_Q }_{L^\infty_E}^{p-1} \norm{\sup_{\substack{Q\in \cd:\\ \pif(Q)=F}} \angles{f}_Q 1_Q }_{L^1_E}\Big)^{1/p}.
\end{split}
\end{equation*}
From the stopping condition \eqref{max_stoppingcondition_lattice}, it follows (see Table \ref{table_stopping conditionss}) that $$\norm{\sup_{\substack{Q\in \cd:\\ \pif(Q)=F}} \angles{f}_Q 1_Q }_{L^\infty_E}\leq 2 \angles{\abs{\bar{M}_\cd f}_E}_F.$$ From the stopping condition \eqref{max_stoppingcondition_average}, it follows (again, see Table \ref{table_stopping conditionss}) that 
$$
\angles{f}_Q=\angles{f_F}_Q \text{ whenever } \pif(Q)=F,$$ where the auxiliary function $f_F$ is defined by $f_F:=f1_{\ef(F)}+\sum_{F'\in\ch(F)} \angles{f}_{F'}1_{F'} 
$ and satisfies  $$\norm{f_F}_{L^\infty_E}\lesssim 2^d \angles{\abs{f}_E}_F.$$  Therefore, from the testing condition \eqref{eq_testingmaximal}, H\"older's inequality together with the identity $(p-1)p'=p$, and the dyadic Carleson embedding theorem (Lemma \ref{lem_carlesonembedding}), it follows that
\begin{equation*}
\begin{split}
&\Big(\sum_{F\in\cf}\norm{\sup_{\substack{Q\in \cd:\\ \pif(Q)=F}} \angles{f}_Q 1_Q }_{L^\infty_E}^{p-1} \norm{\sup_{\substack{Q\in \cd:\\ \pif(Q)=F}} \angles{f}_Q 1_Q }_{L^1_E}\Big)^{1/p}\\
&=\Big(\sum_{F\in\cf}\norm{\sup_{\substack{Q\in \cd:\\ \pif(Q)=F}} \angles{f}_Q 1_Q }_{L^\infty_E}^{p-1} \norm{\sup_{\substack{Q\in \cd:\\ \pif(Q)=F}} \angles{f_F}_Q 1_Q }_{L^1_E}\Big)^{1/p}\\
&\lesssim \mathfrak{M}^{1/p} \Big(\sum_{F\in\cf} \angles{\abs{\bar{M}_\cd f}_E}_F^{(p-1)}\mu(F)^{1/p'}\angles{\norm{f}_E}_F \mu(F)^{1/p}\Big)^{1/p}\\
&\leq  \mathfrak{M}^{1/p} \bigg(\Big(\sum_{F\in\cf} \angles{\abs{\bar{M}_\cd f}_E}_F^{(p-1)p'}\mu(F)\Big)^{1/p'}  \Big(\sum_{F\in\cf}\angles{\norm{f}_E}_F^p \mu(F)\Big)^{1/p}\bigg)^{1/p}\\
&\lesssim_p \mathfrak{M}^{1/p} \norm{\mbar_\cd f}_{L^p_E}^{1/p'} \norm{f}_{L^p_E}^{1/p}.
\end{split}
\end{equation*}
Altogether,
$$
\norm{\mbar_\cd f}_{L^p_E}\lesssim_p  \norm{\mbar_\cd f}_{L^p_E}^{1/p'} (\mathfrak{M}\norm{f}_{L^p_E})^{1/p},
$$
from which the norm estimate follows, by dividing out the factor $\norm{\mbar_\cd f}_{L^p_E}^{1/p'}$.
\end{proof}

\begin{question}[Borderline: Can we omit the assumption that the measure is doubling?]For each (in particular, for non-doubling) locally finite Borel measure $\mu$, is the boundedness of the dyadic lattice maximal operator $\bar{M}^\mu_\cd: L^p_E(\mu)\to L^P_E(\mu)$ characterized by the endpoint direct $L^\infty(\mu)$ testing condition?
\end{question}

\section*{Acknowledgments}
This paper is part of the author's Ph.D. thesis project written under the supervision of Tuomas Hyt\"onen. The author is supported by the European Union through Hyt\"onen's ERC Starting Grant \lq Analytic-probabilistic methods for borderline singular integrals\rq. The author thanks Mikko Kemppainen for communicating him that the boundedness of the lattice maximal operator can be characterized by the existence of an auxiliary function with certain concavity and boundedness properties, and for providing the author with an unpublished manuscript containing the proof of this characterization.

\bibliographystyle{plain}
\bibliography{references_two_weights}
\end{document}